\documentclass[10pt,a4paper]{article} 
\usepackage{amsmath,amssymb,amsthm,amsfonts,amscd,euscript,verbatim, t1enc, newlfont, graphicx, cancel}
\usepackage{hyperref}
\usepackage[all]{xy}
\providecommand{\keywords}[1]{\textbf{\textit{Key words and phrases }} #1}
\providecommand{\subjclass}[1]{\textbf{\textit{2010 Mathematics Subject Classification.}} #1}

\theoremstyle{definition}
\newtheorem{theo}{Theorem}[subsection]

\newtheorem{theore}{Theorem}[section]
\newtheorem{pr}[theo]{Proposition}

 \newtheorem{lem}[theo]{Lemma}
 \newtheorem{coro}[theo]{Corollary}
  
\theoremstyle{remark}
\newtheorem{rema}[theo]{Remark}

\newtheorem{rrema}[theore]{Remark}
\theoremstyle{definition}
\newtheorem{defi}[theo]{Definition}

\newtheorem{prop}[theore]{Proposition}
\newtheorem{defn}[theore]{Definition}
\numberwithin{equation}{subsection}

\newcommand\cu{\underline{C}}
\newcommand\du{\underline{D}}

\newcommand\au{\underline{A}}
\newcommand\aucp{\underline{A}_{\cp}}
\newcommand\bu{\underline{B}}
\newcommand\hu{\underline{H}}

\newcommand\obj{\operatorname{Obj}}
\newcommand\mo{\operatorname{Mor}}
\newcommand\id{\operatorname{id}}
\DeclareMathOperator\adfu{\operatorname{AddFun}}

\DeclareMathOperator\kar{\operatorname{Kar}}
 
 \DeclareMathOperator\cok{\operatorname{Coker}}
\DeclareMathOperator\imm{\operatorname{Im}}
\DeclareMathOperator\co{\operatorname{Cone}}

\DeclareMathOperator\prli{\varprojlim}
\DeclareMathOperator\inli{\varinjlim}
\DeclareMathOperator\hcl{\underrightarrow{\operatorname{hocolim}}} 

\newcommand\hw{{\underline{Hw}}}
\newcommand\hrt{{\underline{Ht}}}

\newcommand\hf{{\underline{HF}}}

\newcommand\chow{\operatorname{Chow}}
\newcommand\wchow{w_{Chow}{}}
\newcommand\wstu{w^{st}}
\newcommand\dm{DM}
\newcommand\dmr{\dm_{R}}
\newcommand\dmk{\dm_{k}}
\newcommand\dmgr{\dm^{gm}_{R}}
\newcommand\dmgk{\dm^{gm}_{k}}
\newcommand\dmgq{\dm^{gm}_{\q}}
\newcommand\chowr{\chow_R}
\newcommand\chowk{\chow_k}

\newcommand\da{DA^{et}(k,k)}
\newcommand\omdr{\mathbf{\Omega}}
\newcommand\omp{\mathbf{\Omega}'}

\DeclareMathOperator\cha{\operatorname{char}}

\newcommand\q{{\mathbb{Q}}}
\newcommand\com{\mathbb{C}}

\newcommand\ql{{\mathbb{Q}_{\ell}}}
\newcommand\zop{{\mathbb{Z}[\frac{1}{p}]}}
\newcommand\z{{\mathbb{Z}}}

 \newcommand\lan{\langle}
\newcommand\ra{\rangle}

\newcommand\ob{^{-1}}
\newcommand\lam{\Lambda}
\newcommand\al{\alpha}

\newcommand\gam{\Gamma}

\newcommand\ns{\{0\}}
\newcommand\ab{\operatorname{Ab}}
\newcommand\abfr{\operatorname{FreeAb}}

\newcommand\cp{\mathcal{P}}
\newcommand\perpp{{}^{\perp}}
\newcommand\opp{^{op}}

\newcommand\shg{SH(G)}
\newcommand\wg{w^G}
\newcommand\wgfin{w^G_{fin}}
\newcommand\macg{\mathcal{M}_G}
\newcommand\emg{\operatorname{EM}_G}
\newcommand\emo{\operatorname{EM}}
\newcommand\egam{\operatorname{EM}^{\gam}}
\newcommand\emz{\operatorname{EM}^{\z}}

\newcommand\wsp{w^{sph}}
\newcommand\hwsp{\hw^{sph}}
\newcommand\shtop{SH}
\newcommand\hsing{H^{sing}}

\newcommand\fil{{\operatorname{Fil}}}
\newcommand\pwcu{\operatorname{Post}_w(\cu)}
\newcommand\cuw{\cu_w}
\newcommand\kw{K_{\mathfrak{w}}}

\newcommand\ca{{\mathcal{A}}}
\newcommand\cacp{{\mathcal{A}_{\cp}}}
\newcommand\psvr{\operatorname{PShv}^R}

\newcommand\tg{\tilde g_0}
\begin{document}

\title
 {On weight complexes,  pure functors, and detecting  weights}
\author{Mikhail V. Bondarko
   \thanks{ 
 The main results of the paper were  obtained under support of the Russian Science Foundation grant no. 16-11-00073.}}\maketitle
\begin{abstract} This paper is dedicated to the study of weight complex functors (defined on triangulated categories endowed with weight structures) and their applications. We introduce {\it pure} (co)homological functors that "ignore all non-zero weights"; these have a nice description in terms of weight complexes. 

 An important example is  the weight structure $\wg$ generated by the orbit category in the $G$-equivariant stable homotopy category $\shg$;  
 the corresponding pure cohomological functors into abelian groups are the Bredon cohomology  associated to Mackey functors ones. 
 Pure functors related to "motivic" weight structures are also quite useful. 

Our results also give some (more) new weight structures and the conservativity of certain weight-exact functors. 
	 Moreover, we prove that certain functors "detect weights", i.e., check whether an object belongs to the given level of the weight filtration.
\end{abstract}


\tableofcontents

 \section*{Introduction}

In this paper we treat several questions related to weight complex functors; the latter are defined on triangulated categories endowed with weight structures (as independently defined by the author and D. Pauksztello).

We give an important definition of {\it pure}   (co)homological functors.\footnote{The relation of pure functors to Deligne's purity of (singular and \'etale) cohomology is recalled in Remark \ref{rpuresd}(3).} 
 Functors of this type have already found interesting applications in several papers (note in particular that the results of  our \S\ref{sdet} are important for the study of Picard groups of triangulated categories in \cite{bontabu}; other interesting pure functors were crucial for \cite{kellyweighomol}, \cite{bachinv},  \cite{bgn}, and \cite{bsoscwhn}). Pure functors can be defined in two distinct ways: for a weight structure $w$ on a triangulated category $\cu$ one can either demand that a (co)homological functor $H$ from $\cu$ into an abelian category $\au$ kills objects whose "weights" are either (strictly) positive or negative, or "reconstruct" $H$ of this type from its values on the {\it heart} $\hw$ of $w$ using the corresponding  weight complex functor (and one obtains a pure functor from any additive functor from $\hw$ into $\au$ using this method).

Now we recall that the original weight complex functor from \cite{gs} has associated certain complexes of Chow motives to varieties over characteristic $0$ fields. We vastly generalize this construction (following \cite{bws}) 
 and obtain a "weakly exact" functor $t:\cu\to \kw(\hw)$ (this is a certain "weak" category of complexes; see Proposition \ref{pwt} and Remark \ref{rwc}(\ref{irwc7}) below) 
 corresponding  to any weight structure $w$.  
These general weight complex functors  are closely related to {\it weight spectral sequences} (that generalize Deligne's ones and 
 also "calculate" the values of pure functors; see Proposition \ref{pwss}). Moreover, the conservativity properties of these functors enable us to prove that a {\it weight-exact functor} (i.e., an exact functor that "respects"  the corresponding weight structures) whose restriction to the heart is full and conservative is also conservative on weight-bounded objects.  Combined with the recent results of J. Ayoub (see Remark \ref{rgap}), this statement 
 implies the conservativity of the $\ql$-\'etale (and de Rham) realization on the category $\dmgq$ of geometric Voevodsky motives over a characteristic $0$ field (see Remark \ref{rayoub}(3)).  We also extend this result to a bigger motivic category.

 Furthermore, we apply our general theory to the case of the {\it spherical} weight structure $\wg$ on the equivariant stable homotopy category $\shg$; here $G$ is a compact Lie group. $\wg$ is generated by the (stable) {\it orbit category}; the latter  consists of spectra of the form form $S_H^0$ (see \S\ref{shg}),  where $H$ runs through closed subgroups of $G$.  
 The corresponding class $\shg_{\wg\ge 0}$ is the class of connective $G$-spectra, the heart $\hw^G$ consists of retracts of coproducts of $S_{H_i}^0$,  whereas the weight complex functor calculates the 
equivariant   ordinary 
 homology with Burnside ring coefficients   $H^G_*$   considered in \cite{lewishur} and \cite{mayeq}, 
  and pure cohomological  functors into abelian groups are representable by  Eilenberg-MacLane $G$-spectra and equal the Bredon cohomology functors corresponding to Mackey functors. Moreover, in the case $G=\{e\}$ (respectively, $\shg=\shtop$) the corresponding {\it  $\wsp$-Postnikov towers} are the  cellular ones in the sense of \cite[\S6.3]{marg}. 

\begin{rrema}\label{rgap}
Unfortunately, the current proof of the main results of Ayoub's \cite{ayoubcon}  contains a gap. Hopefully, it will be closed eventually.
Anyway, the fact that Theorem II of ibid. implies the conservativity of realizations conjecture is non-trivial and interesting for itself. 

Note also that the latter conservativity assertion has several nice applications;  some of them were described in \cite{bcons}.
\end{rrema}

Now we describe the contents of the paper;  
 some more information of this sort can also be found at the beginnings of sections.

In \S\ref{sold} we recall a significant part of the general theory of  weight structures. Moreover, we treat weight complexes more accurately than in 
 \cite[\S3]{bws} (see also 
 Appendices \ref{sdwc}--\S\ref{swhe} for some additional remarks on this matter). Furthermore, we apply 
 our theory to obtain a new theorem on the conservativity of weight-exact functors; our results generalize certain statements from \cite{wildcons}.

In \S\ref{spuredet} we introduce and study {\it pure} (homological and cohomological) functors. These functors are quite important  (cf. \S\ref{stop}  for "topological" examples, whereas several motivic examples were shown to be actual in several recent papers). We relate them to detecting weights (i.e., we prove for a  {\it $w$-bounded below} object $M$ that it belongs to 
 the $n$th level of the weight filtration whenever $H_i^{\ca}(M)=0$ for all $i< n$, where $H_*^{\ca}$ is a certain homological functor). This matter is important for the categorical Picard calculations of \cite{bontabu}. 
We also study ({\it smashing}) weight structures and pure functors that  respect  coproducts.

In \S\ref{sexamples} we recall some statements that allow to reconstruct a weight structure 
 from a subcategory of its heart. These theorems give the existence of so-called Chow weight structures on certain categories of Voevodsky motives; next we discuss the aforementioned motivic conservativity statements. Moreover, we study pure functors and detecting weights for {\it purely compactly generated} (smashing) weight structures.  We also study possible "variations" of {\it weight Postnikov towers} and the corresponding weight complexes (for a fixed object $M$ of $\cu$); as a consequence, we obtain a new existence of weight structure statement along with some more motivic conservativity results.

In \S\ref{stop} we relate our general theory to  the stable homotopy category $\shg$ of $G$-equivariant spectra (for any compact Lie group $G$) along with the "spherical" weight structure $\wg$  (generated by the stable orbit subcategory  of equivariant spheres). We prove that the corresponding pure cohomology is  Bredon one. In the case $G=\{e\}$ 
we prove that singular homology detects weights, and that $\wsp$-Postnikov towers are the {\it cellular} ones in the sense of \cite{marg}. We also discuss the relation of our results to  ({\it adjacent}) $t$-structures, and to the {\it connective stable homotopy theory} as described in \S7 of \cite{axstab}. 

The author is deeply grateful to prof. J.P. May for his very useful answers concerning equivariant homotopy categories, and to the referee for really important comments to the text.  He is  also 
 extremely grateful to the Max Planck Institute 
 in Bonn for the support 
 and  hospitality during the work on this version.

\section{On weight structures:   reminder, weight complexes, and conservativity applications}
\label{sold}

In \S\ref{snotata} we introduce some notation and conventions.

In \S\ref{ssws} we recall some basics on weight structures.

\S\ref{sswc} is dedicated to the theory of weight complex functors. Our treatment of this subject (along with weight Postnikov towers) is more accurate than the original one in \cite{bws} (cf. \S\ref{sdwc}--\ref{swhe} below).

In \S\ref{swss} we recall  the basics of the theory of weight spectral sequences.

In \S\ref{sweap} we apply our theory to obtain an interesting statement on the conservativity of weight-exact functors.

\subsection{Some (categorical) notation }\label{snotata} 

\begin{itemize}
\item	All coproducts in this paper will be small.

\item Given a category $C$ and  $X,Y\in\obj C$  we  write $C(X,Y)$ for  the set of morphisms from $X$ to $Y$ in $C$.

\item For categories $C',C$ we write $C'\subset C$ if $C'$ is a full 
subcategory of $C$.

\item We  say that $D$ is an {\it essentially wide} subcategory of $C$ if $D$ is a full subcategory of $C$ that is equivalent to $C$. Moreover, we will say that $D$ is a {\it skeleton} of $C$ if any two isomorphic objects of $D$ are equal.

\item Given a category $C$ and  $X,Y\in\obj C$, we say that $X$ is a {\it
retract} of $Y$  if $\id_X$ can be 
 factored through $Y$.\footnote{If $C$ is triangulated or abelian, 
then $X$ is a retract of $Y$ if and only if $X$ is its direct summand.}\ 

\item A  (not necessarily additive) subcategory $\hu$ of an additive category $C$ 
is said to be {\it retraction-closed} in $C$ if it contains all retracts of its objects in $C$.

\item  For any $(C,\hu)$ as above the full subcategory $\kar_{C}(\hu)$ of 
 $C$ whose objects are all retracts of  (finite) direct sums of objects $\hu$ in $C$ will be called the {\it retraction-closure} of $\hu$ in $C$; note that this subcategory is obviously additive and retraction-closed in $C$.

\item The {\it Karoubi envelope} $\kar(\bu)$ (no lower index) of an additive category $\bu$ is the category of ``formal images'' of idempotents in $\bu$. Consequently, its objects are the pairs $(A,p)$ for $A\in \obj \bu,\ p\in \bu(A,A),\ p^2=p$, and the morphisms are given by the formula 
$$\kar(\bu)((X,p),(X',p'))=\{f\in \bu(X,X'):\ p'\circ f=f \circ p=f \}.$$ 
 The correspondence  $A\mapsto (A,\id_A)$ (for $A\in \obj \bu$) fully embeds $\bu$ into $\kar(\bu)$.
 Moreover, $\kar(\bu)$ is {\it Karoubian}, i.e.,  any idempotent morphism yields a direct sum decomposition in 
 $\kar(\bu)$.  Recall also that $\kar(\bu)$ is triangulated if $\bu$ is (see \cite{bashli}).

\item The symbol $\cu$ below will always denote some triangulated category; usually it will
be endowed with a weight structure $w$. The symbols $\cu'$ and $\du$ will  also be used  for triangulated categories only. 

\item For any  $A,B,C \in \obj\cu$ we  say that $C$ is an {\it extension} of $B$ by $A$ if there exists a distinguished triangle $A \to C \to B \to A[1]$.

\item A class $D\subset \obj \cu$ is said to be  {\it extension-closed}  if it	is closed with respect to extensions and contains $0$. We  call the smallest extension-closed subclass of objects of $\cu$ that  contains a given class $B\subset \obj\cu$   the  {\it extension-closure} of $B$. 

\item Given a class $D$ of objects of $\cu$ we will write $\lan D\ra$ for the smallest  full retraction-closed triangulated subcategory of $\cu$ containing $D$. We   call  $\lan D\ra$  the triangulated category {\it densely generated} by $D$.    Certainly, this definition can be applied in the case $\du=\cu$.

Moreover, we will say that  $D$ {\it strongly generates} $\cu$ if  $\cu$ equals its own smallest strictly full triangulated subcategory that contains $D$.\footnote{Clearly, this condition is fulfilled if and only if  $ \cu$ equals the extension-closure of $\cup_{j\in \z}D[j]$.}

\item For $X,Y\in \obj \cu$ we will write $X\perp Y$ if $\cu(X,Y)=\ns$. For
$D,E\subset \obj \cu$ we write $D\perp E$ if $X\perp Y$ for all $X\in D,\
Y\in E$. Given $D\subset\obj \cu$ we   write $D^\perp$ for the class $$\{Y\in \obj \cu:\ X\perp Y\ \forall X\in D\}.$$
Dually, ${}^\perp{}D$ is the class $\{Y\in \obj \cu:\ Y\perp X\ \forall X\in D\}$.

\item Given $f\in\cu (X,Y)$, where $X,Y\in\obj\cu$, we  call the third vertex
of (any) distinguished triangle $X\stackrel{f}{\to}Y\to Z$ a {\it cone} of
$f$.\footnote{Recall that different choices of cones are connected by non-unique isomorphisms.}\

\item Below $\au$ will always  denote some abelian category; $\bu$ is an additive category.

\item We  write $C(\bu)$  for the category of (cohomological) complexes over $\bu$;  $K(\bu)$ is its homotopy category. 
 The full subcategory of $K(\bu)$ consisting of bounded complexes  will be denoted by $K^b(\bu)$. 

We  write $M=(M^i)$ if $M^i$ are the terms of a complex $M$.
	
	\item We will say that an additive covariant (resp. contravariant) functor from $\cu$ into $\au$ is {\it homological} (resp. {\it cohomological}) if it converts distinguished triangles into long exact sequences.
	
	For a (co)homological functor $H$ and $i\in\z$ we  write $H_i$ (resp. $H^i$) for the composition $H\circ [-i]$.

\end{itemize}

\subsection{Weight structures: basics}\label{ssws}

\begin{defi}\label{dwstr}

I. A pair of subclasses $\cu_{w\le 0},\cu_{w\ge 0}\subset\obj \cu$ 
will be said to define a weight
structure $w$ on a triangulated category  $\cu$ if 
they  satisfy the following conditions.

(i) $\cu_{w\ge 0}$ and $\cu_{w\le 0}$ are retraction-closed in $\cu$ (i.e., contain all $\cu$-retracts of their objects).

(ii) {\bf Semi-invariance with respect to translations.}

$\cu_{w\le 0}\subset \cu_{w\le 0}[1]$ and $\cu_{w\ge 0}[1]\subset \cu_{w\ge 0}$.

(iii) {\bf Orthogonality.}

$\cu_{w\le 0}\perp \cu_{w\ge 0}[1]$.

(iv) {\bf Weight decompositions}.

 For any $M\in\obj \cu$ there exists a distinguished triangle $$LM\to M\to RM {\to} LM[1]$$
such that $LM\in \cu_{w\le 0} $ and $ RM\in \cu_{w\ge 0}[1]$.

Moreover, if $\cu$ is endowed with a weight structure then we will say that $\cu$ is a {\it weighted} (triangulated) category.
\end{defi}

We will also need the following definitions.

\begin{defi}\label{dwso}
Let $i,j\in \z$; assume that a triangulated category $\cu$ is endowed with a weight structure $w$.

\begin{enumerate}
\item\label{idh}
The full  subcategory $\hw$ of $ \cu$ whose objects are $\cu_{w=0}=\cu_{w\ge 0}\cap \cu_{w\le 0}$  is called the {\it heart} of  $w$.

\item\label{id=i}
 $\cu_{w\ge i}$ (resp. $\cu_{w\le i}$, resp. $\cu_{w= i}$) will denote the class $\cu_{w\ge 0}[i]$ (resp. $\cu_{w\le 0}[i]$, resp. $\cu_{w= 0}[i]$).

\item\label{id[ij]}
$\cu_{[i,j]}$  denotes $\cu_{w\ge i}\cap \cu_{w\le j}$; hence this class  equals $\ns$ if $i>j$. 

$\cu^b\subset \cu$ will be the category whose object class is $\cup_{i,j\in \z}\cu_{[i,j]}$; we  say that its objects are the {$w$-bounded} objects of $\cu$.

\item\label{idbo}
We   say that $(\cu,w)$ is {\it  bounded}  if $\cu^b=\cu$ (i.e., if $\cup_{i\in \z} \cu_{w\le i}=\obj \cu=\cup_{i\in \z} \cu_{w\ge i}$).

\item\label{idbob} We  call $\cup_{i\in \z} \cu_{w\ge i}$ (resp. $\cup_{i\in \z} \cu_{w\le i}$) the class of {\it $w$-bounded below} (resp., {\it $w$-bounded above}) objects of $\cu$.

\item\label{idwe}  Let   $\cu'$ be a triangulated category endowed with  a weight structure $w'$; let $F:\cu\to \cu'$ be an exact functor.

Then $F$ is said to be  {\it  weight-exact} (with respect to $w,w'$) if it maps $\cu_{w\le 0}$ into $\cu'_{w'\le 0}$ and
sends $\cu_{w\ge 0}$ into $\cu'_{w'\ge 0}$. 

\item\label{idrest}
Let $\du$ be a full triangulated subcategory of $\cu$.

We will say that $w$ {\it restricts} to $\du$ whenever the couple $(\cu_{w\le 0}\cap \obj \du,\ \cu_{w\ge 0}\cap \obj \du)$ is a weight structure on $\du$.

\item\label{ilrd}
We will say that $M$ is left (resp., right) {\it $w$-degenerate} (or {\it weight-degenerate} if the choice of $w$ is clear) if $M$ belongs to $ \cap_{i\in \z}\cu_{w\ge i}$ (resp.    to $\cap_{i\in \z}\cu_{w\le i}$).

\item\label{iwnlrd} We  say that $w$ is left (resp., right) {\it non-degenerate} if all left (resp. right) weight-degenerate objects are zero.
\end{enumerate}
\end{defi}

\begin{rema}\label{rstws}

1. A  simple (and still  useful) example of a weight structure comes from the stupid filtration on the homotopy category $K(\bu)$  of complexes over an arbitrary additive  $\bu$ (it can also be restricted to bounded complexes; see Definition \ref{dwso}(\ref{idrest})). 
 We set $K(\bu)_{\wstu\le 0}$ (resp. $K(\bu)_{\wstu\ge 0}$) to be the class of complexes that are
homotopy equivalent to complexes  concentrated in degrees $\ge 0$ (resp. $\le 0$); see Remark 1.2.3(1) of \cite{bonspkar} for more detail. We will use this notation below. 

 The heart of this weight structure is the retraction-closure  of $\bu$  in  $K(\bu)$; hence it is equivalent to $\kar(\bu)$.

2. A weight decomposition (of any $M\in \obj\cu$) is almost never canonical. 

Still for any $m\in \z$ the axiom (iv) gives the existence of a distinguished triangle \begin{equation}\label{ewd} w_{\le m}M\to M\to w_{\ge m+1}M\to (w_{\le m}M)[1] \end{equation}  with some $ w_{\ge m+1}M\in \cu_{w\ge m+1}$ and $ w_{\le m}M\in \cu_{w\le m}$; we  call it an {\it $m$-weight decomposition} of $M$.

 We will often use this notation below even though $w_{\ge m+1}M$ and $ w_{\le m}M$ are not canonically determined by $M$. We  call any possible choice either of $w_{\ge m+1}M$ or of $ w_{\le m}M$ (for any $m\in \z$) a {\it weight truncation} of $M$. Moreover, when we will write arrows of the type $w_{\le m}M\to M$ or $M\to w_{\ge m+1}M$ we  always assume that they come from some $m$-weight decomposition of $M$.

3. In the current paper we use the ``homological convention'' for weight structures; it was previously used in  \cite{wildshim}, \cite{wildcons}, \cite{hebpo}, 
\cite{brelmot}, \cite{bonivan}, \cite{bonspkar},  
  \cite{bokum},  \cite{bgn},  \cite{bkwn}, \cite{bvtr}, and \cite{bpws}, whereas in \cite{bws}, \cite{bger}, and \cite{bontabu} the ``cohomological convention'' was used. In the latter convention the roles of $\cu_{w\le 0}$ and $\cu_{w\ge 0}$ are essentially interchanged, i.e., one
considers  the classes $\cu^{w\le 0}=\cu_{w\ge 0}$ and $\cu^{w\ge 0}=\cu_{w\le 0}$.  Consequently,  a
complex $X\in \obj K(\bu)$ whose only non-zero term is the fifth one (i.e.,
$X^5\neq 0$) has weight $-5$ in the homological convention, and has weight $5$
in the cohomological convention. Thus the conventions differ by ``signs of
weights'';  $K(\bu)_{[i,j]}$ is the class of retracts of complexes concentrated in degrees  $[-j,-i]$. 
 
 We also recall that 
D. Pauksztello has introduced weight structures independently (see \cite{paucomp}); he called them co-t-structures. \end{rema}

\begin{pr}\label{pbw}
Let  $m\le l\in\z$, $M,M'\in \obj \cu$, $g\in \cu(M,M')$. 

\begin{enumerate}
\item \label{idual}
The axiomatics of weight structures is self-dual, i.e., for $\cu'=\cu^{op}$ (so $\obj\cu'=\obj\cu$) there exists the (opposite)  weight structure $w'$ for which $\cu'_{w'\le 0}=\cu_{w\ge 0}$ and $\cu'_{w'\ge 0}=\cu_{w\le 0}$.

\item\label{iort}
 $\cu_{w\ge 0}=(\cu_{w\le -1})^{\perp}$ and $\cu_{w\le 0}={}^{\perp} \cu_{w\ge 1}$.

\item\label{icoprod} $\cu_{w\le 0}$ is closed with respect to all coproducts that exist in $\cu$.

\item\label{iext}  $\cu_{w\le 0}$, $\cu_{w\ge 0}$, and $\cu_{w=0}$ are additive and extension-closed. 

\item\label{isplit} If $A\to B\to C\to A[1]$ is a $\cu$-distinguished triangle  and $A,B,C\in  \cu_{w=0}$ then this distinguished triangle splits, that is,  $B\cong A\bigoplus C$. 

\item\label{igenlm}
The class $\cu_{[m,l]}$ is the extension-closure of $\cup_{m\le j\le l}\cu_{w=j}$.

\item\label{ibond} If $M$ is bounded above (resp. below)  and also left (resp. right) $w$-degenerate then it is zero. 

\item\label{ifact} Assume  $M'\in   \cu_{w\ge m}$. Then any $g\in \cu(M,M')$ factors through $w_{\ge m}M$ (for any choice of the latter object).

Dually, if $M\in   \cu_{w\le m}$ then any $g\in \cu(M,M')$ factors through $w_{\le m}M'$.

\item\label{iwdmod} If $M$ belongs to $ \cu_{w\le 0}$ (resp. to $\cu_{w\ge 0}$) then it is a retract of any choice of $w_{\le 0}M$ (resp. of $w_{\ge 0}M$).

 \item\label{iwd0} 
 If $M\in \cu_{w\ge m}$  then $w_{\le l}M\in \cu_{[m,l]}$ (for any $l$-weight decomposition of $M$). 

Dually, if  $M\in \cu_{w\le l}$   then $w_{\ge m}M\in \cu_{[m,l]}$.

\item\label{icompl} 				For any (fixed) $m$-weight decomposition of $M$ and an $l$-weight decomposition of $M'$  (see Remark \ref{rstws}(2))   $g$ can be extended to a morphism of the corresponding distinguished triangles:
 \begin{equation}\label{ecompl} \begin{CD} w_{\le m} M@>{c}>>
M@>{}>> w_{\ge m+1}M\\
@VV{h}V@VV{g}V@ VV{j}V \\
w_{\le l} M'@>{}>>
M'@>{}>> w_{\ge l+1}M' \end{CD}
\end{equation}

Moreover, if $m<l$ then this extension is unique (provided that the rows are fixed).

\item\label{iwdext} For any distinguished triangle $M\to M'\to M''\to M[1]$  and any 
weight decompositions $LM\stackrel{a_{M}}{\to} M\stackrel{n_{M}}{\to} R_M\to LM[1]$ and $LM''\stackrel{a_{M''}}{\to} M''\stackrel{n_{M''}}{\to} R_M''\to LM''[1]$ there exists a commutative diagram 
$$\begin{CD}
LM @>{}>>LM'@>f>> LM''@>{}>>LM[1]\\
 @VV{a_M}V@VV{a_{M'}}V @VV{a_{M''}}V@VV{a_{M}[1]}V\\
M@>{}>>M'@>{}>>M''@>{}>>M[1]\\
 @VV{n_M}V@VV{n_{M'}}V @VV{n_{M''}}V@VV{n_{M}[1]}V\\
RM@>{}>>RM'@>{}>>RM''@>{}>>RM[1]\end{CD}
$$
in $\cu$ whose rows are distinguished triangles and the second column is a weight decomposition (along with the first and the third one).
\end{enumerate}
\end{pr}
\begin{proof}

Assertions  \ref{idual}-\ref{igenlm} and 
\ref{icompl}--\ref{iwdext} 
 were proved  in \cite{bws} (cf.  Remark 1.2.3(4) of \cite{bonspkar} and pay attention to Remark \ref{rstws}(3) above!). 

To prove assertion \ref{ibond} it suffices to consider the case where $M$ is bounded above  and  left  $w$-degenerate since the remaining case is its dual (see assertion \ref{idual}). Now, these assumptions imply that $M\in \cu_{w\le n}$ and $M\in \cu_{w\ge n+1}$ for any large enough $n\in \z$; hence $M\perp M$, i.e., $M=0$.

Assertion \ref{ifact} follows from assertion \ref{icompl} immediately. Next, assertion \ref{iwdmod} is straightforward from the previous assertion  (applied to the morphism $\id_M$).

Lastly, assertion \ref{iwd0} is an easy consequence of assertion \ref{iext}; cf. Proposition 1.3.3(6) of \cite{bws}.
\end{proof}

\subsection{On weight Postnikov towers and weight complexes}\label{sswc}

To define the weight complex functor we  need the following definitions.

\begin{defi}\label{dfilt}

Let $M\in \obj \cu$.

1. A datum consisting of  $M_{\le i}\in \obj \cu$, $h_i\in \cu(M_{\le i},M)$, and $j_i\in  \cu(M_{\le i},M_{\le i+1})$ for $i$ running through integers will be called a {\it filtration on $M$} if we have $h_{i+1}\circ j_i=h_i$  for all $i\in \z$; we  write $\fil_M$ for this filtration.

A filtration will be said to be {\it bounded} if there exist $l\le m\in \z$ such that $M_{\le i}=0$ for all $i<l$ and $h_i$ are isomorphisms for all $i\ge m$.

2. A filtration as above equipped with distinguished triangles \begin{equation}\label{etpt}  M_{\le i-1}\stackrel{j_{i-1}}{\to}M_{\le i} \stackrel{c_{i}}{\to} M_i\stackrel{e_{i-1}}{\to}  M_{\le i-1}[1]\end{equation}
 for all $i\in \z$ will be called a {\it  Postnikov tower} for $M$ or for  $\fil_M$; this tower will be denoted by $Po_{\fil_M}$.

We  use the symbol $M^p$ to denote  $M_{-p}[p]$. 

3. If $\fil_{M'}=(M'_{\le i}, h'_i, j'_i)$ is a filtration on $M'\in \obj \cu$ and $g\in \cu(M,M')$ then we  call $g$ along with a collection of $g_{\le i}\in \cu(M_{\le i}, M'_{\le i})$  a {\it morphism of filtrations compatible with $g$} if  $g\circ h_i=h'_i\circ g_{\le i}$ and  $j'_i\circ g_{\le i} =g_{\le i+1}\circ j_i$ for all $i\in \z$.

Moreover, 
if we have Postnikov towers for $Po_{\fil_M}$ and $Po_{\fil_{M'}}$ for $\fil_{M}$ and  $\fil_{M'}$, respectively, then a datum consisting of a morphism of filtrations compatible with $g$ along with $g_i:M_i\to M'_i$ will be said to give a Postnikov tower morphism  $Po_{\fil_M}\to Po_{\fil_{M'}}$ if all the 
diagrams of the form 
$$\begin{CD}
 M_{\le i}@>{c_i}>>M_i@>{e_{i-1}}>>M_{\le i-1}[1]\\
@VV{g_{\le i}}V@VV{g_i}V@VV{g_{\le i-1}[1]}V \\
M'_{\le i}@>{c'_i}>>M'_i@>{e'_{i-1}}>>M'_{\le i-1}[1]
\end{CD}$$
are commutative. Lastly, we will write $g^i$ for $g_{-i}[i]$.
\end{defi}

Let us recall a few simple properties of these notions.

\begin{lem}\label{lrwcomp}
1. For a Postnikov tower as above the  morphisms \break $d^i= c_{-i-1}[i+1] \circ e_{-i-1}[i]: M^i\to M^{i+1}$ (for $i\in \z$) give a complex. 

We  call it the {\it complex  associated with} $Po_{\fil_M}$.

2. Any filtration can be completed to a Postnikov tower uniquely up to a non-unique isomorphism,  and any morphism of filtrations extends to a morphism of the corresponding Postnikov towers. 

Moreover, any morphism of Postnikov towers gives a morphism of the associated complexes.

3. If a filtration of $M$ is bounded then $M$ belongs to the extension-closure of  the corresponding $\{M_i\}$. \end{lem}
\begin{proof}
1. We have $d^{i+1}\circ d^i=c_{-i-2}[i+2] \circ (e_{-i-2}[i+1]\circ c_{-i-1}[i+1]) \circ e_{-i-1}[i]=0$. 

2. The existence an uniqueness of these extensions is 
 a straightforward application of the axioms TR 1 and TR 3 of triangulated categories.

Lastly, $g^i$ give a morphism of complexes since all the diagrams of the form 
$$\begin{CD}
 M^i@>{e_{-1-i}[i]}>>M_{\le -i-1}[1+i]@>{c_{-1-i}[i+1]}>>M^{i+1} \\
@VV{g^i}V@VV{g_{\le -i-1}[1+i]}V@VV{g^{i+1}}V \\
M'{}^i@>{e'_{1-i}[i]}>>M'_{\le -i-1}[1+i]@>{c'_{-1-i}[i+1]}>>M'{}^{i+1}
\end{CD}$$
are commutative.

3. Assume that   $M_{\le l}=0$ 
for some $l\in \z$. Then obvious induction yields that the corresponding objects $M_{\le n}$ belong to the extension-closure of $\{M_i\}$ for all $n\ge l$. Thus if  $M_{\le m}\cong M$ for some $m\ge l$ then $M$ belongs to this extension-closure as well.
\end{proof}

Now let us relate these notions with weight structures.

\begin{defi}\label{dwpt}
Assume that $\cu$ is a weighted triangulated category (see Definition \ref{dwstr}). 

1. We  call a filtration (see Definition \ref{dfilt}) $\fil_M$ on $M\in \obj \cu$ a {\it weight filtration} (on $M$) if the morphisms $h_i:M_{\le i}\to M$ yield $i$-weight decompositions for all $i\in \z$ (in particular, 
  $M_{\le i}=w_{\le i}M$; see Remark \ref{rstws}(2)). 

We will call  the corresponding $Po_{\fil_M}$ (see Lemma \ref{lrwcomp}(2)) a {\it weight Postnikov tower} for $M$.

2. $\pwcu$ will denote the category whose objects  are objects of $\cu$ endowed with 
 weight Postnikov towers and whose morphisms are morphisms of Postnikov towers.

$\cuw$ will be the category whose objects are the same as for $\pwcu$ and such that  $\cuw(Po_{\fil_M},Po_{\fil_{M'}})=\imm (\pwcu(Po_{\fil_M},Po_{\fil_{M'}})\to \cu(M,M'))$ (i.e., we kill those morphisms of towers that are zero on the underlying objects).

3. For an additive category $\bu$, complexes $M,N\in \obj C(\bu)$, and morphisms $m_1,m_2\in C(\bu)(M,N)$  we  write $m_1\backsim m_2$ if $m_1-m_2=d_N\circ x+y\circ d_M$ for some collections of arrows $x^*,y^*:M^*\to N^{*-1}$, where $d_M$ and $d_N$ are the corresponding differentials.

We  call this equivalence relation the {\it weak homotopy (equivalence)} one (cf. Remark \ref{rwc}(\ref{irwc2}) below).
\end{defi}

\begin{pr}\label{pwt}
In addition to the notation introduced above (in particular, note that $g\in \cu(M,M')$) assume that $\bu$ is an additive category and $n\in \z$.
\begin{enumerate}
\item\label{iwpt1}
Any choice of $i$-weight decompositions of $M$ for  $i$ running through integers  gives a unique weight filtration on $M$ with $M_{\le i}=w_{\le i}M$, and  $M^i\in \cu_{w=0}$. 


\item\label{iwpt2} Any $g\in \cu(M,M')$ can be extended to a morphism of (any choice of) weight filtrations for $M$ and $M'$, respectively; hence it also extends to a morphism of weight Postnikov towers.

\item\label{iwpt3} The obvious functor $\cuw\to \cu$ is an equivalence of categories.

\item\label{iwhecat} Factoring morphisms in $K(\bu)$ by the weak homotopy equivalence relation yields an additive category $\kw(\bu)$. Moreover, the corresponding full functor $K(\bu)\to \kw(\bu)$ is (additive and) conservative.

\item\label{iwhefu}
Let $\ca:\bu\to \au$ be an additive functor, where $\au$ is any abelian category. Then for any $B,B'\in \obj K(\bu)$ any pair of weakly homotopic morphisms $m_1,m_2\in C(\bu)(B,B')$  induce equal morphisms of the homology $H_*((\ca(B^i)))\to H_*((\ca(B'^i)))$.

\item\label{iwhefun}
Sending an object of $\cuw$ into the complex  given by Lemma \ref{lrwcomp}(1) and a morphism of Postnikov towers into the corresponding $(g^i)$ (see Definition \ref{dfilt}(3)) yields a well-defined additive functor 
$t=t_w:\cuw\to \kw(\hw)$.

We  call this functor the {\it weight  complex} one. 
 We
will often write $t(M)$ for $M\in \obj \cu$ (resp. $t(g)$) assuming that some weight Postnikov tower for $M$ (resp. a lift of $g$ to $\cuw$) is chosen; we  say that $t(M)$ is {\it a choice of a weight complex} for $M$.

\item\label{irwcsh} $t\circ [n]_{\cuw}\cong [n]_{\kw(\hw)}\circ t$, where  $[n]_{\cuw}$ and  $[n]_{\kw(\hw)}$ are the obvious shift by $[n]$ (invertible) endofunctors of the categories  $\cuw$ and $\kw(\hw)$, respectively.

\item\label{iwcons} Assume that $M$ is bounded above (resp. below). Then $M\in \cu_{w\le n}$ (resp. $M\in \cu_{w\ge n}$) if and only if $t(M)$ belongs to $K(\hw)_{\wstu\le n}$ (resp. to $K(\hw)_{\wstu\ge n}$; cf. Remark \ref{rwc}(\ref{irwco}) below).

\item\label{iwcex} 
If $M\stackrel{g}{\to} M' \stackrel{f}{\to} M''$ is a distinguished triangle in $\cu$ then for any choice of $t(M)$ and $t(M'')$
 there exists a compatible choice of $(t(g),t(f))$  (so, the domain of this $t(g)$ is the chosen $t(M)$ and the target of  $t(f)$ is $t(M'')$) along with their lifts to $K(\hw)$ that can be completed to a distinguished triangle in $K(\hw)$.

\item\label{irwcons}
 If $M\in \cu_{w\ge n}$ for some $n\in \z$ then there exists  a weight Postnikov tower $Po_M$ with $M_{\le i}=w_{\le i}M=0$ for all $i<n$; consequently, $M^i=0$ for $i>-n$.

Moreover,  if $M\in \cu_{w\le  n}$ then we can take $M_{\le i}=w_{\le i}M=M$ for all $i\ge n$ to obtain $M^i=0$ for $i<-n$.

Consequently, if $M$ is left or right weight-degenerate then $t(M)=0$ for the corresponding choice  of a weight Postnikov tower for $M$; the composition of  the  obvious  embedding $\hw\to \cuw$ with $t$ is isomorphic to the obvious embedding $\hw \to \kw(\hw)$.

\item\label{iwch}  Assume  $N\in \cu_{w=0}$, and $N$ is a retract of $M$. Then $N$ is a also retract of  
 the object  $M^0$ for any choice of $t(M)=(M^i)$.

\item\label{iwcfunct} Let $\cu'$ be a triangulated category endowed with a weight structure $w'$; let
 $F:\cu\to \cu'$ be a weight-exact functor. Then $F$ is compatible with a naturally defined functor $F_w:\cuw\to \cuw'$, and the composition $t'\circ F_w$ 
 equals $\kw(\hf)\circ t$, where $t'$ is the weight complex functor corresponding to $w'$, and the functor $\kw(\hf):\kw(\hw)\to \kw(\hw')$ is the obvious $\kw(-)$-version of the restriction $\hf:\hw\to \hw'$ of $F$ to $\hw$.

\item\label{iwc2342} Assume that $t(g)=(g^i)$, and the arrows $g^i$ come from an actual $\pwcu$-morphism (between the corresponding weight Postnikov towers) that is compatible with $g$. Then any  family $({\tilde g}^{i})$ such that $({\tilde g}^{i})=(g^i)$ in $K(\hw)$ (that is,   $(\tilde g^{i})$ is homotopy equivalent to  $(g^{i})$)  extends to a morphism of these towers that is compatible with $g$ as well. 
\end{enumerate}
\end{pr}
\begin{proof}
Taking into account our definitions (cf. also Lemma  \ref{lrwcomp}(2)), assertions \ref{iwpt1}--
\ref{iwcons}  
 easily follow from the results of \cite{bws}; see Lemma 1.5.1(1,2) 
Lemma 3.1.4(I.1,II.1), Remark 3.1.7(2), Theorem 3.2.2(II), and Theorem 3.3.1(IV,I) of ibid., respectively.

Moreover, in Appendix \ref{sdwc}  
 below 
we 
 prove assertions \ref{iwpt1}--\ref{iwpt3}, \ref{iwhefun},  \ref{iwcons},  and \ref{iwcex}, whereas  assertions \ref{iwhecat} and \ref{iwhefu} are contained in Proposition \ref{pwwh}. 

 Next, all the statements in assertions \ref{irwcsh} and  \ref{irwcons} follow from our definitions immediately.

Assertion \ref{iwch} is easy as well. We can set $t(N)=N$ and $t(M)=(M^i)$. Applying $t$ to the fact that $N$ is a retract of $M$ we obtain that $\id_N$ factors through $t(M)$ in $\kw(\hw)$. Looking at the corresponding 
$C(\hw)$-morphisms we deduce that $N$ is  a retract of $M^0$ indeed (cf. the end of assertion \ref{iwcex}).  


To prove assertion \ref{iwcfunct} 
 we note that weight-exact functors send weight Postnikov towers and their morphisms in $\cu$ into that in $\cu'$. Hence we can define the functor $F_w$ in the obvious way, and the equality $t'\circ F_w=\kw(\hf)\circ t$ is automatic.

 Assertion \ref{iwc2342} is rather  technical and will not be applied in this paper. For this reason, we place this proof in  \S\ref{sirwc} as well.
\end{proof}

\begin{rema}\label{rwc}
\begin{enumerate}
\item\label{irwco}
Combining parts \ref{iwhefun}, \ref{iwpt3}, and \ref{iwhecat} of our proposition we obtain that all possible choices of $t(M)$ are homotopy equivalent. 
 
Thus the assumptions that $t(M)\in K(\hw)_{\wstu\ge n}$ and $t(M)\in K(\hw)_{\wstu\le n}$ do not depend on the choice of $t(M)$.

More generally, we will not mention choices when speaking of those properties of $t(M)$ that are preserved by $K(\hw)$-isomorphisms. 

\item\label{irwc2} The weak homotopy equivalence relation was introduced  in \S3.1 of \cite{bws} independently from the earlier and closely related notion of {\it absolute homology}; cf. Theorem 2.1 of \cite{barrabs}.

\item\label{irwc7} 
 $t$ can "usually" be "enhanced" to an  exact functor $t^{st}:\cu\to K(\hw)$; see Corollary 3.5 of \cite{sosnwc} and \S6.3 of \cite{bws}. However, the author currently does not know how to obtain exact weight complex functors for the Chow weight structures studied in \cite{bokum}.

Note also that to obtain a strong weight complex functor that is compatible with $t$ one 
  has to change the signs of differentials either in Lemma \ref{lrwcomp}(1)  or in the category $K(\hw)$ (see \S2 and Definition 5.7 of \cite{schnur}). However, since this matter does not appear to affect any of the applications of weight complexes known to the author, the reader may probably ignore it.

\item\label{irwc51} Our definition of weight complexes is not (quite) self-dual; see Remark \ref{rwcbws}(\ref{irwc52}) below for more detail. Note also that we do 
 use any octahedra  in our definitions (in contrast to Definition 1.5.8 of \cite{bws} and Definition 5.7 of \cite{schnur}); yet we will need an octahedral diagram in Appendix \ref{sdwc}  below.

Lastly, our weight Postnikov towers generalize {\it cellular towers} of \cite{marg}; see Theorem \ref{top}(\ref{itopcell},\ref{itopskel}) below and Theorem 4.2.3(7) of  \cite{bkwn}.
\end{enumerate}
\end{rema}

\subsection{Weight spectral sequences: reminder}\label{swss}

Let us recall weight spectral sequences for cohomology and homology. 

\begin{pr}\label{pwss}
Assume that $\cu$ is a weighted category.  

1. If $H$ is a cohomological functor from $\cu$ into an abelian category $\au$ then for any $M\in \obj \cu$  and any 
 choice of $t(M)$  there exists a spectral sequence $T=T_w(H,M)$ with $E_1^{pq}=H^{q}(M^{-p})$,  such that $M^i$ and the boundary morphisms of $E_1(T)$ come from this $t(M)$.

Moreover, $T_w(H,M)$ is $\cu$-functorial  in $M$ 
 starting from $E_2$; in particular, these levels of $T$ do not depend on any choices. 

 Furthermore, $T$ converges to $H^{p+q}(M)$ whenever $H$ kills $\cu_{w\ge i}$ and $\cu_{w\le -i}$ for $i$ large enough, or if $M$ is bounded above and $H$ kills  $\cu_{w\le -i}$ for $i$ large enough.

2. Dually, for a homological $H':\cu\to \au$, any $M\in \obj \cu$, and any 
 choice of $t(M)$   there exists a spectral sequence $T=T_w(H',M)$ with $E_1^{pq}(T)=H'_{-q}(M^{p})$, such that the boundary morphisms of $E_1(T)$ come from $t(M)$ as well. 

Moreover,  $T_w(H',M)$ is $\cu$-functorial  in $M$ 
   starting from $E_2$, and  converges to $H'_{-p-q}(M)$ whenever $H'$ kills $\cu_{w\ge i}$ and $\cu_{w\le -i}$ for $i$ large enough. 
\end{pr}
\begin{proof}
I. This is (most of) Theorem 2.4.2  of \cite{bws} (yet take into account Remark \ref{rstws}(3)!).  

2. See Theorem 2.3.2  of \cite{bws} (yet note that the numeration of homology in the current paper is opposite 
 to that in loc. cit.!).
\end{proof}

\begin{rema}\label{re1ss} 
Recall that  weight spectral sequences in \cite[\S2]{bws} were constructed using weight Postnikov towers of objects.

 Now, all  choices of $t(M)$ (that come for weight Postnikov towers for this object) are homotopy equivalent (see Remark \ref{rwc}(\ref{irwco})); yet it is not quite true that all complexes that are $K(\hw)$-isomorphic to a given $t(M)$ can be "realized" by  $w$-Postnikov towers (see Remark \ref{realiz} below). However, this subtlety is not really essential for the theory of weight spectral sequences.

Let us justify this claim for cohomological weight spectral sequences (coming from part 1 of our proposition). For {\bf any} complex $(M'{}^i)\in \obj K(\hw)$  that is homotopy equivalent to $t(M)$ the homology of the complex $H^{q}(M'{}^{-*})$ is isomorphic to that for  $H^{q}(M{}^{-*})$ (for our fixed choice $t(M)=(M^i)$). 
 Hence for any $(M'{}^i)$ of this sort the corresponding $E_1$-level of $T_w(H,M)$ "fits" with higher levels of $T$ computed using any fixed  weight Postnikov tower for $M$.
\end{rema}

\subsection{An application to "detecting weights" by weight-exact functors}\label{sweap}  

Proposition \ref{pwt} easily implies that certain weight-exact functors are conservative.  We will discuss some motivic applications of the following theorem in Remark \ref{rayoub}(2,3) below.

\begin{theo}\label{twcons}
Let $\cu$ and $\cu'$ be  triangulated categories endowed with weight structures $w$ and $w'$, respectively; let $F:\cu\to \cu'$ be a weight-exact functor (see Definition \ref{dwso}(\ref{idwe})).

Assume that the induced functor $\hf:\hw\to \hw'$ is full and conservative,  $M$ is an object of $\cu$, and $n\in \z$.

1. Suppose that $M$ is $w$-bounded above (resp. below). Then $F(M)$ belongs to $ \cu'_{w'\le n}$ (resp. to $ \cu'_{w'\ge n}$) if and only if  $M$ belongs to $ \cu_{w\le n}$ (resp. to $ \cu_{w\ge n}$).

2. Assume that $M$ is $w$-bounded above (resp., below) and $F(M)=0$. Then $M$ is right (resp. left) $w$-degenerate.

Moreover, if $M$ is $w$-bounded  then it is zero.

3.  Suppose that $M$ is $w$-bounded above; for any $i\in \z$ such that $M\in \cu_{w\le i}$ and any $N\in \cu_{w=i}$ assume that the homomorphism $\cu(M,N)\to \cu'(F(M),F(N))$ induced by $F$ is zero.

Then $M$ is right $w$-degenerate.
\end{theo}
\begin{proof}
1. The "if" implication is immediate from the definition of weight-exactness. So we verify the converse implication.

It clearly suffices to consider the case where  $M$ is $w$-bounded above and $F(M)$ belongs to $ \cu'_{w'\le n}$ since the remaining case is its dual (see Proposition \ref{pbw}(\ref{idual})). We will write $t(M)=(M^i)$ for a choice of a weight complex for $M$; the boundary morphism of this complex will be denoted by $d^i$. 

According to Proposition \ref{pwt}(\ref{iwcons}) it suffices to verify that $t(M)\in K(\hw)_{\wstu\le n}$ (see Remark \ref{rstws}(2)). Now choose the minimal integer $m\ge n$ such that  $t(M)\in K(\hw)_{\wstu\le m}$;  we should prove that $m=n$.

Next, we can assume that $M^i=0$ for $i<-m$ (see Proposition \ref{pwt}(\ref{irwcons})). By Proposition \ref{pwt}(\ref{iwcfunct}), the complex $t_{w'}(F(M))=(M'{}^i)$ can be obtained  from $(M^i)$ by means of termwise application of $F$; hence $M'{}^i=0$ for $i<-m$ (and this choice of $t_{w'}(F(M))$) as well. 

Now suppose that  $m-1\ge n$. Then $(M'{}^i)\in  K(\hw')_{\wstu\le m-1}$; hence the morphism $d'{}^{-m}$ is  split monomorphic, i.e., there exists $s'\in \hw'(M'{}^{1-m},M'{}^{-m})$ such that $s'\circ d'{}^{-m}=\id_{M'{}^{-m}}$. Since $\hf$ is full and conservative and $\hw$ is a retraction-closed subcategory of $\cu$, the morphism $d^{-m}$ is isomorphic to the split monomorphism $M^{-m}\to (M^{-m}\bigoplus C)\cong M^{1-m}$ according to Lemma \ref{lsplit} below (here $C$ is some object of $\hw$).   Thus  $t(M)$ actually belongs to $ K(\hw)_{\wstu\le m-1}$ and we obtain a contradiction as desired. 

2. If $F(M)=0$ then $F(M)$ belongs to $\cu_{w\le m}$ and also to $\cu_{w\ge m}$ for all $m\in \z$. Hence the assertion follows from the previous one immediately.

Next, if $M$ is right or left degenerate and   also $w$-bounded then 
 $M=0$ 
  according to Proposition \ref{pbw}(\ref{ibond}).

3. To prove that $M$ is right $w$-degenerate it clearly suffices to verify for any $m\in \z$ that $M$ belongs to $ \cu_{w\le m-1}$ whenever it belongs to $ \cu_{w\le m}$. We assume the latter and choose an $m-1$-weight decomposition triangle $$w_{\le m-1}M\to M\to w_{\ge m}M\to w_{\le m-1}M[1].$$ We will write $N$ for $w_{\ge m}M$; note that $N$ belongs to $\cu_{w=m}$ according to Proposition \ref{pbw}(\ref{iwd0}). 
 Since $F$ is weight-exact, the corresponding triangle  
\begin{equation}\label{efwd}
F(w_{\le m-1}M)\to F(M)\stackrel{z}\to F(N)\to F(w_{\le m-1}M)[1]
\end{equation}
is an $m-1$-weight decomposition of $F(M)$. 

Similarly to the proof of assertion 1, we choose $t(M)=(M^i)$  to satisfy $M^i=0$ for $i<-m$.  Then  Proposition \ref{pwss}(1) easily implies  that $\cu(M,N)$ is the quotient of $\hw(M^{-m},N[-m])$ by the corresponding image of $\hw(M^{1-m},N[-m])$. Moreover, since the functor $F$ is weight-exact, we can take $t_{w'}(F(M))=(F(M^i))$ (see Proposition \ref{pwt}(\ref{iwcfunct})); thus the group $\cu(F(M),F(N))$ is the corresponding quotient of $\hw'(F(M^{-m}),F(N[-m]))$. Since $\hf$ is full, it follows that  the homomorphism $\cu(M,N)\to \cu'(F(M),F(N))$ is surjective under our assumptions; thus $\cu'(F(M),F(N))=0$. We obtain that the morphism $z$ in (\ref{efwd}) is zero. Hence the object $F(M)$ is a retract of $F(w_{\le m-1}M)$; thus it belongs to $ \cu'_{w'\le m-1}$. Applying assertion 1 we conclude that $M$ belongs to $ \cu_{w\le m-1}$ indeed. 
\end{proof}

Thus it remains to prove the following statement.

\begin{lem}\label{lsplit}
Let $G:\bu\to \bu'$ be a full and conservative additive functor (between additive categories), 
 and $h\in \bu(M,N)$ for some objects $M,N\in \obj \cu$. 

1. 
 If $G(h)$ is split injective then  $h$ is split injective as well.

2. Suppose that $\bu$ is a full retraction-closed subcategory of a triangulated category $\cu$ and 
 $h$ is split injective. Then $N$ can be presented as the direct sum  $M\bigoplus N'$ for some $N'\in \obj \bu$ so that 
 $h=\id_M\bigoplus 0:M\to N$. 
\end{lem}
\begin{proof}
1. Since $G(h)$ splits, there exists $s'\in \bu'(G(N),G(M)$ such that $s'\circ G(h)=\id_{G(M)}$. Since $G$ is full, there exists $s\in \bu(N,M)$ such that $s'=G(s)$. Since $G$ is conservative and $G(s\circ h)=\id_{F(M)}$, the composition $s\circ h$ is an automorphism $c$ of $M$. Thus the morphism $h$ is split by $c\ob \circ s$.

2. Since $\cu$ is a triangulated category, $h$ can be presented in the  desired form in $\cu$. Since $\bu$ is retraction-closed in $\cu$, we obtain that the corresponding object $N'$ actually belongs to $\bu$.
\end{proof}

\begin{rema}\label{rwcons} 
1. Respectively, one may say that any functor $F$ as in our theorem detects weights of $w$-bounded objects, i.e., looking at $F(M)$ one can find out whether $M\in \cu_{w\ge n}$ and whether $M\in \cu_{w\le n}$.  This property of $F$ was called {\it $w$-conservativity} in \cite{bachinv}; see Proposition 17 and Theorem 22 of ibid.

2. 
The  bounded ("moreover") 
 part of Theorem \ref{twcons}(2) substantially generalizes Theorem 2.5 of \cite{wildcons}, where  $\hw$ was assumed to be {\it semi-primary} (in the sense of  \cite[Definition 2.3.1]{andkahn}) and Karoubian. Moreover, part 1 of our theorem obviously implies Theorem 2.8(a--c) of ibid. (where  $\hw$ was assumed to be  semi-primary as well). 

3. Proposition 3.2.1 of \cite{bkwn} also gives a generalization of part d of loc. cit.; however, a somewhat stronger assumption on $\hf$ is imposed (cf. Corollary \ref{consb}(2) below). Under this condition other parts of our theorem are 
 extended to objects that are not necessarily $w$-bounded either above or below.  An important tool for obtaining (unbounded) results of this sort  was the theory of {\it morphisms killing weights (in a range)}; we do not treat this notion in the current paper.

4. Even though the conditions of the theorem 
 seem to be somewhat restrictive, it appears to be rather difficult to weaken them. Obviously, the "heart-conservativity" assumption in it cannot be dropped (cf. Remark \ref{rdetect}(4) below). 

It is an interesting question 
to what degree the fullness condition can be weakened. Looking at the proof of the theorem one can easily see that  instead of assuming that $\hf$ is (conservative and) full it suffices to assume the existence of an additive functor $G:\hw'\to \bu$ such that the composition $G\circ \hf$ satisfies this condition. In particular, here one can take $G$ to be the functor that kills the {\it radical} morphism ideal of $\hw'$ (see Definition 1.4 of \cite{wildshim}).   
\end{rema}

\section{On pure functors and detecting weights}\label{spuredet}

In \S\ref{spured} we define pure functors as those that kill all weights except $0$, and prove that they can be expressed in terms of weight complexes.

In \S\ref{sdet} we  study conditions ensuring that a pure functor detects weights of objects. The results of this section are important for the study of Picard groups of weighted triangulated categories   carried over in \cite{bontabu}.

In \S\ref{smash} we prove a rich collection of properties of {\it smashing} weight structures (these are weight structures 
"coherent with (small) coproducts") along with those pure functors that preserve coproducts.

\subsection{Pure functors: equivalent definitions}\label{spured}

Let us define an important class of (co)homological functors from weighted categories. 

\begin{defi}\label{dpure}
Assume that  $\cu$ is endowed with a  weight structure $w$.

We will say that a (co)homological functor $H$ from $\cu$ into an abelian category $\au$ is {\it $w$-pure} (or just pure if the choice of $w$ is clear) if $H$ kills both $\cu_{w\ge 1}$ and $\cu_{w\le -1}$.
\end{defi}

Now we give an explicit description of pure functors.

\begin{theo}\label{tpure}
1. Let $\ca:\hw\to \au$ be an additive functor, where $\au$ is an abelian category. Choose a weight complex $t(M)=(M^j)$ for each $M\in \obj \cu$, and denote by $H(M)=H^{\ca}(M)$ the zeroth homology of the complex $\ca(M^{j})$. Then $H(-)$ yields a homological functor that 
does not depend on the choices of weight complexes.  Moreover, the assignment $\ca\mapsto H^\ca$ is natural in $\ca$. 

2. The correspondence $\ca\to H^\ca$ is an equivalence of categories between the following (not necessarily locally small) categories of functors: $\adfu(\hw,\au)$ and the category of pure homological functors from $\cu$ into $\au$.

3. Assume that $w$ is bounded. Then a (co)homological functor $H$ from $\cu$ into $\au$ is $w$-pure if and only if it annihilates $\cu_{w=i}$ for all $i\neq 0$.

\end{theo}
\begin{proof}
1.  Proposition \ref{pwt}(\ref{iwpt3},\ref{iwhefu},\ref{iwcex}) immediately implies that for any additive $\ca:\hw\to \au$  the functor $H^{\ca}$ is homological and does not depend on the choices of weight complexes. 

2. Firstly, let us prove that any functor of the form  $H^{\ca}$ is pure. For any $M\in \cu_{w\ge 1}\cup \cu_{w\le -1}$ we can choose $t(M)$ with $M^0=0$ according to Proposition \ref{pwt}(\ref{iwcons}); thus $H^{\ca}(M)=0$ in this case.

To finish the proof of the statement it suffices to verify that a $w$-pure functor can be functorially recovered from its restriction to $\hw$. This is immediate from assertion 1 combined with Proposition \ref{pwss}(2).

3. Immediate from our definitions combined with Proposition \ref{pbw}(\ref{igenlm}).
\end{proof}

\begin{rema}\label{rpuresd} 
1. The definition of pure functors is obviously self-dual (cf. Proposition \ref{pbw}(\ref{idual}), i.e., $H$ is pure (co)homological if and only if the functor from $\cu$ into $\au\opp$ obtained from $H$ by means of inversion of arrows is pure cohomological (resp. homological).

Combining this observation with our theorem we obtain that the correspondence $(w,\ca)\mapsto H^{\ca}$ is self-dual as well (cf. Remark \ref{rwcbws}(\ref{irwc52}) below).

We will also use the following notation: for a contravariant additive functor $\ca'$ from $\hw$ into $\au$ we will write $H_{\ca'}$ for the corresponding cohomological pure functor; it clearly sends $M$ as above into the zeroth homology of the complex $\ca'(M^{-j})$.

2. Immediately from the definition of pure functors, a representable functor $\cu(-,M)$ is pure if and only if $M$ belongs to $ (\cu_{w\ge 1}\cup \cu_{w\le -1})\perpp$.

3. The author is using the term "pure" due to the relation of pure functors to Deligne's purity of cohomology. 

To explain it we recall that various categories of Voevodsky motives are endowed with so-called Chow weight structures; we will say more on them in Remarks \ref{rayoub} and \ref{rwchow}(1) below.
 Now, for any $r\in \z$ the $r$th level of the Deligne's weight filtration both of singular and \'etale cohomology of motives clearly kills $\chow[i]$ for all values of $i$ except one (and the remaining value of $i$ is either $r$ or $-r$ 
depending on the choice of the convention for Deligne's weights).\footnote{Certainly, singular (co)homology (of motives) is only defined if the base field $k$ is a subfield of complex numbers, 
 whereas the Deligne's weight filtration on \'etale (co)homology can be defined (at least) if $k$ is a finitely generated field. For both of these cases the comparison of the corresponding weight factors  with the ones computed in terms of $w_{\chow}$ is carried over in  
Theorem 3.5.4(2) of \cite{bsoscwhn} (cf. also Remark 2.4.3 of \cite{bws}). Note also that 
 in \cite[\S3.4,3.6]{brelmot} certain "relative perverse" versions of these weight calculations were discussed.}\  Thus  
(the corresponding shifts of) Deligne's pure factors of (singular and \'etale) cohomology are pure with respect to $w_{\chow}$.

4. Note  however that in the context of equivariant stable homotopy categories (see Theorem \ref{tshg} below) pure (co)homology theories are usually called {\it ordinary} or {\it classical} ones. Functors of this sort were  essentially introduced by Bredon (see condition (4) in \S I.2 of \cite{bred} and Remark \ref{rshg}(\ref{iruniv}) below), and the spectral sequence argument used for the proof of  Theorem \ref{tpure}(2) is rather similar to that applied in \S IV.5 of ibid.

5. Part 1 of our theorem was applied in \cite{bontabu} (see Proposition 3.5(ii) of ibid.).\footnote{Since \cite{bontabu} was written earlier than the current paper, it actually referred to other texts of the author.  Note however that all the weight structure pre-requisites for ibid. are gathered and described more accurately in the current paper.}
\end{rema}

We will also need the following easy statement in succeeding papers.

\begin{lem}\label{lsubc}
Assume that $\au'$ is a strict abelian subcategory of $\au$, i.e., $\au'$ is its strictly full subcategory that contains the $\au$-kernel and the $\au$-cokernel of any morphism in $\au'$.

Then  an additive functor $\ca:\hw\opp\to \au$  factors through $\au'$ if and only if the values of the pure functor $H_{\ca}$ (see Remark \ref{rpuresd}(1)) belong to $\au'$. \end{lem}
\begin{proof}
 Obvious.
\end{proof}

\subsection{Detecting weights via pure functors}\label{sdet}
Let us prove that pure functors can be used to  detect weights (cf. Remark \ref{rwcons}(1)); these results are crucial for \cite{bontabu}.

It will be convenient for us to use the following definitions.

\begin{defi}\label{dbh}
Let $\hu$ be a (not necessarily additive) subcategory of an additive category $\bu$.

1. We  call the full additive subcategory of $\bu$ whose objects are the retracts of  all  (small) coproducts of objects of $\hu$ in $\bu$ (we take all coproducts that exist in $\bu$) the {\it coproductive hull} of $\hu$ in $\bu$.

2. Let  $\hu'$ be the category  of "formal coproducts" of objects of $\hu$, i.e.,  the objects of 
 $\hu'$ are of the form $\coprod_i P_i$ for (families of) $P_i\in \obj \hu$ and $\hu'(\coprod M_i,\coprod N_j)=\prod_i (\bigoplus_j \hu(M_i,N_j))$. Then we will call $\kar(\hu')$ the {\it formal coproductive hull} of $\hu$.
\end{defi}

\begin{rema}\label{rbh}
Obviously, the coproductive hull of $\hu$ in $\bu$, the category $\hu'$ mentioned in our definition, and the formal coproductive hull of $\hu$ are additive categories. Moreover, there exist fully faithful functors between the retraction-closure of $\hu$ in $\bu$  and both aforementioned "coproductive hulls" of $\hu$. 
\end{rema}

\begin{pr}\label{pdetect}
1. Let $\ca:\hw\to \au$ be an additive functor (where $\au$ is an abelian category; cf. Theorem \ref{tpure}) and assume that the following conditions are fulfilled.

(i) the image of $\ca$ consists of  $\au$-projective objects only;

(ii)  if an $\hw$-morphism $h$ is not split surjective (i.e., it is not a retraction) then $\ca(h)$ is not split surjective as well.

Then for any $M\in \obj \cu$, $n\in \z$, and $H=H^\ca$ we have the following:  $M$ is $w$-bounded below and $H_i(M)=0$ (see \S\ref{snotata}) for all $i< n$ if and only if  $M\in \cu_{w\ge n}$.

2. Assume that $\hw$ is an $R$-linear category, where $R$ is a commutative unital ring (certainly, one may take $R=\z$ here),  $\hu$ is a full small subcategory of $\hw$, $\au$ is the category of $R$-linear functors from $\hu\opp$ into the category of $R$-modules (i.e., the objects and morphisms of $\au$ respect the $R$-module structure on morphism groups), $\ca$ is the corresponding Yoneda-type functor, i.e., $\ca(M)$ sends $N\in\obj \hu$ into the $R$-module $\bu(N,M)$ for any $M\in \obj \hw$; suppose in addition that the image of $\ca$ is contained in the  coproductive hull of the image of its restriction to $\hu$ in $\au$. 

Then $\ca$  fulfils condition (i) of assertion 1.

3. Assume that $\ca$ is full and conservative. Then it fulfils condition (ii) of assertion 1.

\end{pr}
\begin{proof}
1. The "if" part of the statement is very easy (for any $\ca$); just combine the definition of $H^\ca$ with  Proposition \ref{pwt}(\ref{irwcons}). 

Now we prove the converse application; we argue similarly to the proof of Theorem \ref{twcons}.

So, we choose the maximal integer $m\le n$ such that  $M\in \cu_{w\ge m}$, choose $t(M)=(M^i)$ such that $M^i=0$ for $i>-m$. According to Proposition \ref{pwt}(\ref{iwcons}), we have $t(M)\notin  K(\hw)_{\wstu\ge m+1}$; applying Lemma \ref{lsplit}(2) (in the dual form) we obtain that the boundary morphism $d^{-m-1}:M^{-m-1}\to M^{-m}$ is not split surjective.
 Applying our assumption (ii) on $\ca$ we obtain that the morphism $\ca(d^{-m-1})$ does not split as well. Since $\ca(M^{-m})$ is projective in $\au$, it follows that $\cok\ca(d^{-m-1})\cong  H_{m}(M)\neq 0$. Thus $m\ge n$, and we obtain the result in question.

2. This statement is much easier than its formulation. Objects of $\hu$ obviously become projective in the category $\au$
(see Remark \ref{rabver}(2) below). Since coproducts and retracts of projective objects are projective, all elements of $\ca(\obj \hw)$ are projective in $\au$ as well.

3. This is just a particular case of Lemma \ref{lsplit}(1). 
\end{proof}

\begin{rema}\label{rabver}
1. Part 1 of our proposition is really similar to Theorem \ref{twcons}(1), and it may be called its "homological functor version". Its advantage as a weight detection method is that pure functors are somewhat easier to construct than weight-exact ones (thanks to Theorem \ref{tpure}; cf. also Corollary \ref{cbontabu}(2) below).   In particular, one can easily deduce Theorem \ref{twcons} from Proposition \ref{pdetect} using certain functors similar to the ones that we will now describe.

2. In the current section we will only apply Proposition \ref{pdetect}(2) (in Corollary \ref{cbontabu}(2)) in the simple case where $\hu$ is an essentially wide small subcategory of $\hw$ (see \S\ref{snotata}; consequently, we assume that $\hw$ is essentially small and then we can take $\hu$ to be a skeleton of $\hw$); in this case it follows immediately both from Lemma 5.1.2 of \cite{neebook}  and from  Lemma 8.1 of \cite{vbook}. 

We have formulated the general case of our proposition for the sake of applying it in Theorem \ref{thegcomp}  below. So we give some more detail for this setting.

Choose  an essentially wide small subcategory $\hu_0$ of   the retraction-closure of $\hu$ in $\cu$.   It is easily seen that the category 
 $\au$ is equivalent to the category $\au'$ of $R$-linear functors from $\hu_0\opp$ into $R$-modules. 
 The category $\au'$ is obviously Grothendieck abelian. Moreover,  the  projectives in it  are precisely the objects of the coproductive hull of the image of $\hu_0$ in $\au'$ (see Definition \ref{dbh}(1)) and there are enough of them according to Lemma 8.1 of \cite{vbook}. Note also that  this hull  coincides with the coproductive hull of the image of $\hu_0$ in $\au'$.
\end{rema}

Let us now deduce Propositions 4.6 and 5.2  of \cite{bontabu} from our proposition.

\begin{coro}\label{cbontabu}
1. Let $\ca:\hw\to \au$ be a full additive conservative functor whose target is semi-simple. Then for a $w$-bounded object $M$ of $\cu$ we have $M\in \cu_{w=0}$ if and only if  $H^\ca_i(M)=0$ for all $i\neq 0$.

2. The assumptions of Proposition \ref{pdetect}(1) are fulfilled whenever  $\hu$ is  an essentially wide small subcategory of the category $\hw$, whereas $R$  and $\ca$ are as in part 2 of that proposition.  
\end{coro}
\begin{proof}
1. Once again, the "if" part of the statement immediately follows from Proposition \ref{pwt}(\ref{irwcons}).  

To  prove the converse implication we note that our assumptions on $\cu,\ca$, and $M$ are self-dual (see  Proposition \ref{pbw}(\ref{idual}) and Remark \ref{rpuresd}(1)). Thus is suffices to prove that $M\in \cu_{w\le 0}$. According to Proposition \ref{pdetect}(1), for this purpose it remains to verify that the image of $\ca$ consists of projective objects only. The latter is automatic since $\au$ is semi-simple.

2. This is very easy. Proposition \ref{pdetect}(2) implies that all objects in the image of $\ca$ are projective in $\au$. Thus it remains to note that the functor $\ca$ is a full embedding according to the Yoneda lemma; hence it is also conservative.  
\end{proof}

\begin{rema}\label{rdetect}
1. Immediately from Lemma \ref{lsplit}(1) (that we have just applied) condition (ii) of  Proposition \ref{pdetect} is  fulfilled both for $\ca$ and for  the opposite functor $\ca^{op}:\hw^{op}\to \au^{op}$ whenever $\ca$ is a full 
conservative functor. In particular, it suffices to assume that $\ca$ is a full embedding.

Hence it may be useful to demand (in addition to assumption (i) of the proposition) that the image of $\ca$ consists of injective objects only (cf. Corollary \ref{cbontabu}(1)). 

2. The objects in the essential image of the functor provided by Corollary \ref{cbontabu}(2)  may be called {\it purely $R$-representable homology}. Since they are usually not injective in  $\psvr(\bu)$, a dual construction may be useful for checking whether $M\in \cu_{w\le -m}$.

3. The boundedness assumption in Proposition \ref{pdetect} cannot be dropped (unless certain additional restrictions are imposed). 

Indeed, take $\bu$ to be the category of free finitely generated $\z/4\z$-modules, $\cu=K(\bu)$, and $w=\wstu$. Then the complex $M=\dots \stackrel{\times 2}{\to}\z/4\z\stackrel{\times 2}{\to}\z/4\z\stackrel{\times 2}{\to}\z/4\z\stackrel{\times 2}{\to}\dots$ does not belong to $\cu_{\wstu\ge i}$ for any $i\in \z$,  whereas its purely $\z$-representable homology (as well as the $\z/4\z$-representable one) is obviously zero in all degrees.

4. Condition (ii)  of Proposition \ref{pdetect}  is certainly necessary. Indeed, if $h\in \mo (\hw)$ does not split whereas $\ca(h)$ does, then one can easily check that $\co(h)\in \cu_{w\ge 0}\setminus \cu_{w\ge 1}$ and $H_i(\co(h))=0$ for $i\neq 1$. 
\end{rema}

\subsection{On smashing weight structures and pure functors respecting coproducts}\label{smash}

It is currently well known (in particular, from \cite{neebook}) that the existence of all  (small) coproducts  
  is a very reasonable assumption on a ("big") triangulated category. Now we relate it to weight structures and pure functors. 

\begin{defi}\label{djsmash}
1. We will say that a triangulated category $\cu$ is {\it smashing} whenever it is closed with respect to 
 coproducts.

2. We  say that a weight structure $w$ on $\cu$ is {\it  smashing} if  the class $\cu_{w\ge 0}$ is closed with respect to  $\cu$-coproducts  (cf. Proposition \ref{pbw}(\ref{icoprod})). 

3. We will say  that a full strict triangulated subcategory $\du\subset \cu$ is  {\it localizing}  whenever it is closed with respect to $\cu$-coproducts. Respectively, we  call the smallest localizing  subcategory of $\cu$ that contains a given class $\cp\subset \obj \cu$  the {\it  localizing subcategory of $\cu$ generated by $\cp$}.  \end{defi}


Let us prove several simple properties of smashing weight structures.

\begin{pr}\label{ppcoprws}
Let $w$ be a  smashing weight structure (on a smashing triangulated category $\cu$), and  $i,j\in \z$.  Then the following statements are valid. 

\begin{enumerate}
\item\label{icopr1} The classes $\cu_{w\le j}$, $\cu_{w\ge i}$, and $\cu_{[i,j]}$ are closed with respect to (small) $\cu$-coproducts. 

\item\label{icoprhw} In particular, the category $\hw$ is closed with respect to $\cu$-coproducts, and the embedding $\hw\to \cu$ preserves coproducts. 

\item\label{icopr2} Coproducts of  weight decompositions are weight decompositions.

\item\label{icopr3}  Coproducts of weight Postnikov towers are weight Postnikov towers.

\item\label{icopr4} The categories $\cuw$ and $\kw(\hw)$ are closed with respect to coproducts, and the functor $t$ preserves coproducts.

\item\label{icopr5} Assume that $\au$ is an  AB4 abelian category. Then  pure functors $\cu\to \au$ 
 that respect coproducts are exactly the functors of the form  $H^\ca$ (see Theorem \ref{tpure}), where $\ca:\hw\to \au$ is an additive functor preserving coproducts. Moreover, this correspondence is an equivalence of  (possibly, big) 
 categories. 

\item\label{icopr5p} Assume that $\au'$ is an  AB4* abelian category. Then pure cohomological  
functors  from $\cu$ into $\au'$ that convert coproducts into products are exactly those of  the form $H_{\ca'}$ (see Remark \ref{rpuresd}(1)),  where $\ca':\hw^{op}\to \au'$ is an additive functor that sends $\hw$-coproducts into products.

\item\label{icopr5b} Assume that $\cu$ satisfies the following {\it Brown representability} property: any cohomological functor from $\cu$ into $\ab$ that converts $\cu$-coproducts into products of groups is representable in $\cu$.

Then the pure functor $H_{\ca'}$ (as above) is representable if and only if $\ca'$ is a contravariant additive functor from $\hw$ into $\ab$ that  converts $\hw$-coproducts into products of groups. Thus the corresponding Yoneda  functor embeds the category of functors satisfying 
 this condition into $\cu$.

\item\label{icopr7} Let $\du$ be the localizing subcategory of $\cu$ generated by a class of objects $\cp$, and assume that for any 
 $P\in \cp$ a choice of (the terms of)   its weight complex $t(P)=(P^k)$  is fixed. Then for any $M\in \obj \du$ 
its weight complex $t(M)$ belongs to 
 the localizing subcategory $K$ of  $K(\hw)$ (see  assertion \ref{icoprhw} and Remark \ref{rwc}(\ref{irwco})) generated by all of the $P^k$. In particular, if $\cp\subset \cu_{w=0}$ then $t(M)$ belongs to 
 the localizing subcategory of  $K(\hw)$  generated by $\cp$.

Moreover, any element of $\cu_{w=0}\cap \obj \du$ 
 belongs to the coproductive hull of $\{P^k\}$ in $\hw$ (see Definition \ref{dbh}(1)).

\item\label{icopr7p} 
For  $\cp$ and $\du$  as in the previous assertion assume that   for any  $P\in \cp$ and any $k\in\z$ some choices of  $w_{\le k}P$ and of $w_{\ge k}P$   are fixed.   Then  for any $D\in\obj \du$ 
there exists a choice of $w_{\le 0}D$ (resp. of $w_{\ge 0}D$)  that belongs to the smallest  class $D_1$ (resp. $D_2$)  of  objects of $\cu$ that is closed with respect to extensions and coproducts and contains ($0$ and) the corresponding objects $(w_{\le k}P)[-k]$ (resp.  $(w_{\ge k}P)[-k]$) for all $P\in \cp$ and $k\in\z$. Moreover,  $\cu_{w\le 0}\cap \obj \du\subset D_1$ and $\cu_{w\ge 0}\cap \obj \du\subset D_2$. 

\item\label{icoprhcl} For a sequence of objects $Y_i$  of $\cu$ for $i\ge 0$ and maps $\phi_i:Y_{i}\to Y_{i+1}$  
 we consider  $C=\coprod Y_i$ and the morphism $a:\oplus \id_{Y_i}\bigoplus \oplus (-\phi_i): C\to C$.

Then a cone $Y$ of $a$ (that is a {\it countable homotopy colimit} of $Y_i$ as defined in  \cite{bokne};  respectively, we will write  $Y=\hcl Y_i=\hcl_{i\ge 0} Y_i$) is right (resp. left) weight-degenerate whenever for all $i\ge 0$ we have  $Y_i\in\cu_{w\le -i}$ (resp.  $Y_i\in\cu_{w\ge i}$). 
\end{enumerate}
\end{pr}
\begin{proof}
\begin{enumerate}
\item The class $\cu_{w\le j}$ is closed with respect to $\cu$-coproducts according to Proposition \ref{pbw}(\ref{icoprod}). Next,  $\cu_{w\ge i}$ is closed with respect to $\cu$-coproducts immediately from Definition \ref{djsmash}(1). It clearly follows that $\cu_{[i,j]}$ possesses this property as well.

\item Immediate from the previous assertion.

\item If the triangles $L_i\to M_i\to R_i\to L_i[1]$ are  weight decompositions of certain $M_i\in \obj \cu$ then the triangle $$\coprod L_i\to \coprod M_i\to \coprod R_i\to \coprod L_i[1]$$ is distinguished according to Remark 1.2.2  of \cite{neebook}. Hence this triangle is a weight decomposition of $\coprod M_i$ according to assertion \ref{icopr1}.

\item Straightforward from 
assertions \ref{icopr1} and \ref{icopr2}.

\item Immediate from assertions  \ref{icopr1} and \ref{icopr3}.

\item Recall that pure functors are exactly those of the type $H^\ca$, and the correspondence $\ca\to H^{\ca}$ is functorial; see   Theorem \ref{tpure}(2). Next,   if $H$ respects coproducts then its restriction to $\hw$ also does according to  assertion  \ref{icoprhw}. Conversely, if $\ca$ preserves coproducts then $H^\ca$ also does according to assertion \ref{icopr4}.

 \item Similarly to assertion \ref{icopr5}, 
 $\ca'$ sends coproducts into products if $H_{\ca'}$ does by  
  assertion  \ref{icoprhw}. Conversely, if $\ca'$  sends coproducts into products then $H_{\ca'}$ also does according to assertion \ref{icopr4}.

\item This is an obvious combination of assertions \ref{icopr5p} and \ref{icoprhw} (cf. also 
 assertion \ref{icopr5}).

\ref{icopr7},\ref{icopr7p}.
Consider the class  $W$ of those $D\in \obj \cu$  that satisfy the following conditions: $t(D)$ (considered as an object of $K(\hw)$) belongs to 
 $K$,  and  for any $m\in \z$ there exist a choice of $(w_{\le m}D)[-m]$ that belongs to $D_1$  and a choice of $(w_{\ge m}D)[-m]$ that belongs to $D_2$. 

Obviously, $W$ is closed with respect to shifts ($[1]$ and $[-1]$). Moreover, $W$ is closed with respect to coproducts according to assertion  \ref{icopr2} and \ref{icopr4}. Furthermore, Proposition \ref{pbw}(\ref{iwdext})  along with Proposition \ref{pwt}(\ref{iwcex}) imply that $W$ is extension-closed. Thus $W$ is the class of objects of a localizing subcategory of $\cu$. Since $W$ contains $\cp$, we obtain that it  contains $\obj \du$ as well.

 Hence for any $M\in \obj \du$ there exists a choice of $t(M)$ 
 that belongs to $K$; note also that in the case $\cp\subset \cu_{w=0}$ one can take $\{P_k\}=\cp$. Next, any object of $K$ is clearly $K(\hw)$-isomorphic to a complex whose terms are coproducts of $P^k$. Thus if $M$ also belongs to $\cu_{w=0}$ then 
  $M$ belongs to the coproductive hull of $\{P^k\}$ in $\hw$  by Proposition \ref{pwt}(\ref{iwch}) applied to the object $t(M) $ of $K(\hw)$ (here we take the stupid weight structure on $K(\hw)$). This finishes the proof of assertion \ref{icopr7}.  

Lastly, if 
 $M\in \obj \du$ belongs to $ \obj \du \cap \cu_{w\le 0}$ (resp. to $ \obj \du \cap \cu_{w\ge 0}$) then the existence of $w_{\le 0}M$ that belongs to $D_1$ (resp. of $w_{\ge 0}M$ that belongs to $D_2$) implies that $M$ belongs to the retraction-closure of $D_1$ (resp. of $D_2$) according to  Proposition \ref{pbw}(\ref{iwdmod}). 
 Now, $D_1$ and $D_2$ are retraction-closed in $\cu$ according to Corollary 2.1.3(2) of \cite{bsnew}; this concludes the proof of assertion \ref{icopr7p}.  

\ref{icoprhcl}. Recall that  $Y=\hcl_{i\ge 0} Y_i\cong \hcl_{i\ge 0} Y_{i+j}$ for any integer $j\ge 0$ according to Lemma 1.7.1 of \cite{neebook}. Since the classes $\cu_{w\le -j}$ and $\cu_{w\ge j}$ are closed with respect to extensions and coproducts (see assertion \ref{icopr1} and Proposition \ref{pbw}(\ref{iext})), we obtain that  $Y$ belongs to $\cu_{w\le 1-j}$ (resp.  $Y_i\in\cu_{w\ge j}$) whenever $Y_i\in\cu_{w\le -i}$ (resp.  $Y_i\in\cu_{w\ge i}$) for all $i\ge 0$, i.e., we obtain the result in question. 
\end{enumerate}
\end{proof}

\begin{rema}\label{ral}
1. In all the parts of our proposition  one can replace arbitrary small coproducts by coproducts of less than $\al$ objects, 
 where $\al$ is any regular infinite cardinal; cf. Proposition \ref{pral} below.

Note however that the corresponding version of 
 Proposition \ref{ppcoprws}(\ref{icopr5b}) appears to be vacuous since the corresponding 
 modification  of the Brown representability condition cannot be fulfilled for a non-zero category $\cu$.

2. 
On the other hand the Brown representability property 
is known to hold for several important classes of triangulated categories (thanks to the foundational results of A. Neeman and others). In particular, it suffices to assume that either $\cu$ or $\cu\opp$ is  generated by a set of {\it compact} objects  (see Definition \ref{dcomp2} below) as its own localizing subcategory.	

3. Some of these statements can be generalized to so-called {torsion theories}; cf. Propositions 2.4(5) and  3.2(2,4) of \cite{bvt}.
\end{rema}

Now let $\al$ be an infinite {\it regular} cardinal number, i.e., $\al$ cannot be presented as a sum of less than $\al$ cardinals that are less than $\al$.

\begin{pr}\label{pral}
Assume that $\cu$ closed with respect to $\cu$-coproducts of less than $\al$ objects  and $\cu_{w\le 0}$ is closed with respect to $\cu$-coproducts of this cardinality.

Then the following statements are valid.

1.  The category $\hw$ is closed with respect to $\cu$-coproducts of less than $\al$ objects,  and the embedding $\hw\to \cu$ preserves these coproducts.

2.  The categories $\cuw$ and $\kw(\hw)$ are closed with respect to coproducts of less than $\al$ objects, and the functor $t$ preserves   coproducts of this sort.

3. Assume that $\cu$ equals its own minimal strict full triangulated subcategory that is closed  with respect to coproducts of less than $\al$ objects and contains a class of objects $\cp\subset \obj \cu$, and assume that for any 
 $P\in \cp$ a choice of (the terms of)   its weight complex $t(P)=(P^k)$ along with $w_{\le k}P$ and  $w_{\ge k}P$ for all $k\in \z$ are fixed. 

Then any element of $\cu_{w=0}$   is a retract of a coproduct  of a family  of $P^k$ of  cardinality less than $\al$. Moreover,  for any object $M$ of $\cu$ there exists a choice of $w_{\le 0}M$ (resp. of $w_{\ge 0}M$)  that belongs to the smallest  class  of  objects of $\cu$ that is closed with respect to extensions and coproducts of  cardinality less than $\al$, and contains ($0$ and) the corresponding objects $(w_{\le k}P)[-k]$ (resp.  $(w_{\ge k}P)[-k]$) for all $P\in \cp$ and $k\in\z$.
	\end{pr}
\begin{proof}
The proofs can be obtained from that of Proposition \ref{ppcoprws}(\ref{icoprhw}, \ref{icopr4}, \ref{icopr7},\ref{icopr7p}) simply by replacing arbitrary small coproducts by coproducts of  cardinality less than $\al$ in all occurrences.
\end{proof}

\begin{rema}\label{rpralz}
1. Certainly, for any $\cu$ and $w$ one can take $\al=\aleph_0$. However, parts 1 and 2 of our proposition do not say anything new in this case. On the other hand, part 3 gives a non-trivial statement; note that in this case $\cu$ is strongly generated by $\cp$ (see \S\ref{snotata}).

2. 
 One can also easily prove that any element $N$ of $\cu_{w=0}$ is a retract of  the coproduct  of a finite family  of the corresponding $P^k$ whenever  $\cu$ is densely  generated by $\cp$; see \S\ref{snotata} and Proposition \ref{pwt}(\ref{iwch}).   

 Let us also make the following observation: if $P\in \cp$ is a $w$-bounded object then we can take $P^k=0$ for almost all values of $k$; see Proposition \ref{pwt}(\ref{irwcons}).
\end{rema}

\section{On "explicit" weight structures and pure functors} \label{sexamples}

In this section we recall two earlier statements on the construction of weight structures "with a given heart", prove a new result of this sort, and describe motivic examples to these assertions.

In \S\ref{sconstrneg} we recall that any {\it connective} densely generating subcategory $\hu$ of $\cu$ gives a weight structure on  it.  Combining this statement with Theorem \ref{twcons} we obtain a conservativity result that does not mention weight structures.  Moreover, we discuss Chow weight structures on categories of Voevodsky motives,  and demonstrate that Theorem \ref{twcons}  allows to deduce the conservativity of  
  the \'etale realization on the category $\dmgq$ of $\q$-linear geometric motives  from Theorem II of \cite{ayoubcon}.

In \S\ref{spcgw} we recall that one can obtain a smashing weight structure on (a smashing triangulated category) $\cu$ from a connective compactly generating subcategory $\hu$ of $\cu$. We also study pure functors and detecting weights for weight structures of this sort. Moreover, we prove that the Chow weight structure $\wchow$ on the category "big" motivic category $\dmr$ is degenerate (if $R$ is not torsion and the base field $k$ is "big enough"). 
 
In \S\ref{subtlety} we study in detail the relation between different choices of weight complexes and weight Postnikov towers for a fixed object $M$ of $\cu$. This allows us to construct some new  weight structures; in particular, we obtain 
 a new conservative weight-exact motivic functor.

\subsection{Constructing  weight structures  starting from connective subcategories, and the conservativity of motivic functors}\label{sconstrneg}


\begin{defi}\label{dcomp1}
Let $\hu$ be a full subcategory of a triangulated category $\cu$.

We will say that $\hu$ is {\it connective} (in $\cu$) if $\obj \hu\perp (\cup_{i>0}\obj (\hu[i]))$.\footnote{ In earlier texts of the author connective subcategories were called {\it negative} ones. Moreover, in several papers (mostly, on representation theory and related matters) a connective subcategory satisfying certain additional assumptions was said to be {\it silting}.}
\end{defi}

First we recall a statement that allows to construct all bounded weight structures (cf. Proposition \ref{pbw}(\ref{igenlm})).

\begin{pr}\label{pexw}
Assume that $\cu$ is densely generated by its connective additive subcategory $\bu$.

1. Then there exists a unique weight structure $w$ on $\cu$ whose heart contains $\bu$. 
Moreover, this weight structure is bounded, $\hw=\kar_{\cu}\bu$, and $\cu_{w\ge 0}$ (resp. $\cu_{w\le 0}$) is the smallest class of objects that contains $\bu[i]$ for all $i\ge 0$ (resp. $i\le 0$), is extension-closed and retraction-closed in $\cu$.

2. Let $w'$ be a weight structure on a triangulated category $\du$. Then an exact functor $F:\cu\to \du$ is weight-exact with respect to $(w,w')$ (for $w$ as above) if and only if it sends $\bu$ inside the heart $\hw'$.
\end{pr}
\begin{proof}
1. This is  essentially Corollary 2.1.2 of \cite{bonspkar}; see also Theorem 5.5 of \cite{mendoausbuch}.

2. This is an immediate consequence of the description $w$ given in assertion 1 along with the fact that both $\cu'_{w'\ge 0}$ and  $\cu'_{w'\le 0}$ are retraction-closed and extension-closed in $\cu'$; cf. also  Lemma 2.7.5 of \cite{bger}.
\end{proof}

Let us combine this statement with Proposition \ref{pdetect} to obtain a conservativity result that does not mention weight structures.

\begin{coro}\label{consb}
1. Let $F:\cu\to \cu'$ be an exact functor; assume that there exists a connective additive subcategory $\bu$ of $\cu$ that densely generates it and such that the full subcategory $\bu'$ of $\cu'$ whose object class equals $F(\obj \bu)$ is connective (in $\cu'$), whereas the restriction of $F$ to $\bu$ is full and conservative. 
 
Then $F$ is conservative itself, i.e., it does not kill non-zero objects.

2. The  conservativity condition in assertion 1 is fulfilled if  all endomorphisms of objects of $\bu$ that are killed by $F$ are nilpotent.  \end{coro}
\begin{proof}
1. Obviously, we can   assume that the category $\cu'$ is densely generated by its subcategory $\bu'$, and we will do so. Thus Proposition \ref{pexw} yields the following: $\cu$ and $\cu'$ are endowed with bounded weight structures $w$ and $w'$, respectively, such that $\hw=\kar_{\cu}\bu$,  $\hw'=\kar'_{\cu}\bu$, and $F$ is weight-exact with respect to them. Now we verify that the remaining assumptions of Theorem \ref{twcons} are fulfilled in this setting.

Since all objects of $\hw$ are retracts of objects of $\bu$, the fullness of the restriction of $F$ to $\bu$ implies the fullness of the
corresponding functor $\hf:\hw\to \hw'$. Thus it remains to verify that $\hf$ is conservative. For an $\hw$-morphism $g:M\to N$ we should check that it is invertible whenever $F(g)$ is. Now the fullness assumption gives us the existence of a morphism $h\in \hw(N,M)$ such that $F(h)$ is the inverse to $F(g)$. It remains to prove that the endomorphisms $g\circ h$ and $h\circ g$ 
 are automorphisms. 
Thus it suffices to verify that for any $Q\in \cu_{w=0}$  a morphism $p\in \cu(Q,Q)$ is an automorphism whenever $F(p)$ is. Next, $Q$ is a retract of an object of $S$ of $\bu$, and since $\cu$ is a triangulated,  $S\cong Q\bigoplus R$ 
 (for some $R\in \obj \cu$; actually, $R$ belongs to $\cu_{w=0}$ as well). Thus  $F(p\bigoplus \id_{R})$ is an automorphism; since the restriction of $F$ to $\bu$ is conservative, it follows that $p\bigoplus \id_{R}$ is an automorphism as well. Hence $p$ is an automorphism indeed.

2. A well-known easy fact; see Remark 3.1.5 of \cite{bws}.
\end{proof}

\begin{rema}\label{rayoub}
 Let us describe the relation of our results to  Voevodsky motives.

1. Let $k$ be a perfect field of characteristic $p$ ($p$ may be a prime or zero); let $R$ be a $\zop$-algebra, where we set $\zop=\z$ if $p=0$. Denote by $\dmr$ the smashing category of $R$-linear Voevodsky motives over $k$ (see  \S4.2 of \cite{degmod} or \S1.1 of \cite{cdint}). The category  $\dmr$ contains the additive category $\chowr$ of $R$-linear Chow motives that is connective in it (see Remark 3.1.2 of \cite{bokum},  Remark \ref{rwchow}(1), or  Proposition \ref{paydegen} below). Thus we obtain a {\it Chow} weight structure on the subcategory $\dmgr$ densely generated by $\chowr$ in $\dmr$. Moreover, this weight structure extends to a smashing weight structure on the whole $\dmr$; see Proposition \ref{paydegen} below. 

This method of constructing the Chow weight structure originates from  \S6.5 of \cite{bws} (cf. \cite{bzp} for the case $p>0$); it was carried over to relative motives in \cite{hebpo} and \cite[\S2.1]{brelmot}, whereas in \S2.3 of ibid., \S2.1 of \cite{bonivan}, and \cite{bokum} some other methods were described. 

Moreover, applying Theorem \ref{tpure}(3) we obtain that a (co)homological functor from $\dmgr$ is pure whenever it kills $\chowr[i]$ for all $i\neq 0$. Furthermore, Theorem \ref{thegcomp}(\ref{itnp2}) below gives a similar characterization of those pure functors from $\dmr$ that respect coproducts.

2. Let us now describe an interesting weight-exact functor from $\dm$ that is quite relevant for  \cite{ayoubcon}. 

In this treatise the case $p=0$ and $R=k$ was considered. The category $\dmk$ is equivalent to the category $\da$ (of \'etale $k$-linear motives). Now let us use the symbol $\omp$ for the  truncated de Rham spectrum $\tau_{\ge 0}\omdr$ (see Theorem II of ibid.); note that it is a highly structured ring spectrum with respect to the  model structure on $\da$ that was considered in ibid. 

Thus we can take $\cu'$ to be the derived category of highly structured $\omp$-modules in $\da$, $\cu=\dmgk$,
and take $F$ to be the restriction to $\dmgk$  of  the "free module" functor $-\otimes \omp:\da\to \cu'$.

Next, we can take $\bu$ either to be the subcategory $\chowk\subset \da\cong \dmk$ (of $k$-linear Chow motives) or the category of twists of motives of smooth projective varieties by $-(i)[2i]=\lan i \ra$ for $i\in \z$.  The images of these motives in $\cu'$ give a connective category whose Karoubi envelope is the $k$-linear category of $k$-motives up to algebraic equivalence (cf. the formula (xxviii) of \cite{ayoubcon}; these statements are based on simple cohomological dimension and Poincare duality arguments along with  Remark 7.6 of \cite{blog}). 

It remains to note that the nilpotence assumption of Corollary \ref{consb}(2) easily follows from Corollary 3.3 of \cite{voevnilp}. We obtain that the functor $F$ is conservative (and detects weights; cf. Remark \ref{rwcons}(1)). 

3. Furthermore, Theorem II of  \cite{ayoubcon} 
  appears to  imply that for any $M\in \obj \dmgk$ whose de Rham cohomology is zero the functor $F$ kills all morphisms from $M$ into $\chowk[i]$ for any $i\in \z$ such that  $M\perp \cup_{j>i}\obj \chowk[j]$. Applying Theorem \ref{twcons}(3) we obtain that $M=0$ (cf. Remark \ref{rwcons}(2)). Thus loc. cit. implies that de Rham cohomology is conservative on $\dmgk$ (this Conjecture II of ibid.). This is a very interesting observation that is substantially stronger than Theorem I of ibid. This conservativity conjecture was shown to have very interesting implications in \cite{bcons};  yet see Remark \ref{rgap} above. 

4. We also recall that certain  functors that are pure with respect to  the corresponding Chow weight structures were crucial for the recent papers \cite{kellyweighomol}, \cite{bachinv},  \cite{bontabu}, and \cite{bsoscwhn}. All of these pure functors were defined in terms of   weight complex functors (cf. Theorem \ref{tpure}). Moreover, in \cite{bgn} functors that are pure with respect to certain {\it Gersten weight structures} are considered. 
\end{rema}

\subsection{On purely compactly generated weight structures}\label{spcgw}

Now we pass to the study of  a particular family of smashing weight structures.
In this subsection we will always assume that $\cu$ is smashing.

\begin{defi}\label{dcomp2}
1. An object $M$ of $\cu$ is said to be {\it compact} if  the functor $H^M=\cu(M,-):\cu\to \ab$ respects coproducts. 

2. We will say that $\cu$ is {\it compactly generated} by $\cp\subset \obj \cu$ if $\cp$ is a {\bf set} of compact objects 
that generates  $\cu$ as its own localizing subcategory (see Definition \ref{djsmash}(3)).
\end{defi}

First we recall the following well-known statement.

\begin{lem}\label{lcg}
Let $\cp$ be a set of compact objects of $\cu$. 
 Then $\cp$ compactly generates $\cu$ if and only if $(\cup_{i\in \z}\cp[i])^\perp=\ns$. \end{lem}
\begin{proof}
This is (a part of) \cite[Proposition 8.4.1]{neebook}.\end{proof}

\begin{theo}\label{thegcomp}
Let $\hu$ be a connective  subcategory of $\cu$ such that $\obj \hu$ compactly generates $\cu$, and 
  $\cp=\obj \hu$. Then the following statements are valid.
\begin{enumerate}

\item\label{itnpb} $\cu$ satisfies the  Brown representability property (see Proposition \ref{ppcoprws}(\ref{icopr5b})).

\item\label{itnp1}
There exists a unique smashing weight structure $w$ on $\cu$ such that $\cp\subset \cu_{w=0}$; this weight structure is left non-degenerate.

\item\label{itnp1d} For this weight structure $\cu_{w\le 0}$ (resp. $\cu_{w\ge 0}$) is the smallest subclass of $\obj \cu$ that is closed with respect to coproducts, extensions, and contains $\cp[i]$ for $i\le 0$ (resp. for $i\ge 0$), and $\hw$ is the coproductive hull of $\hu$ in $\cu$ (see Definition \ref{dbh}(1)).

Moreover, $\cu_{w\ge 0}=(\cup_{i<0}\cp[i])^{\perp}$. 

\item\label{itnp2} Let $H$ be a cohomological  functor from $\cu$ into an abelian category $\au$ that converts all small coproducts into products. Then it is pure if and only if it kills $\cup_{i\neq 0}\cp[i]$.

\item\label{itnwe} Let $F:\cu\to \du$ be an exact functor that respects coproducts, where $\du$ is (a triangulated category)     endowed with a smashing weight structure $v$. Then $F$ is weight-exact if and only if it sends $\cp$ into $\du_{v=0}$.

\item\label{itnp3} The category $\hrt\subset \cu$ of $w$-pure representable functors from $\cu$ (so, we identify an object of $\hrt$ with the functor from $\cu$ that it represents)  is equivalent to the category $\au_{\cp}$ of 
 additive contravariant functors from $\hu$ into $\ab$ (i.e., we take those functors that respect the addition of morphisms). 

Moreover, $\aucp$ (and so also $\hrt$) is Grothendieck abelian, has enough projectives, and  an injective cogenerator. 
Restricting functors representable by elements of $\cp$ to $\hw$ one obtains a full embedding $\hw\to \aucp$ whose essential image is the subcategory of projective objects of $\aucp$.

\item\label{itnpr} $w$ restricts (see Definition \ref{dwso}(\ref{idrest})) to a bounded weight structure on the subcategory of compact objects of $\cu$, and the  heart of this restriction is the  retraction-closure of $\hu$ in $\cu$ (so, we consider only retracts of finite coproducts $\coprod P_i$ for $P_i\in \cp$).
\end{enumerate}
\end{theo}
\begin{proof}

Assertion \ref{itnpb} is a particular case of  Proposition 8.4.2 of \cite{neebook}.

Assertions \ref{itnp1}--\ref{itnwe}  follow from  Corollary  2.3.1 and Lemma 2.3.3 of \cite{bsnew} easily (see Remark 2.3.2(2) of ibid.). 

\ref{itnp3}. Since objects of $\hu$ are compact in $\cu$, $\hw$ is naturally equivalent to the formal coproductive hull of $\hu$ (see Definition \ref{dbh}(2)). Thus $\hw$ embeds into the category $\aucp$. As we have noted in Remark \ref{rabver}(2) (we take $R=\z$ in that remark), the essential image of this embedding is the subcategory of projective objects of $\aucp$. Moreover, the category $\aucp$ has enough projectives; since  it is Grothendieck abelian, it also possesses an injective cogenerator.

It remains to prove that $\hrt$ is equivalent to $\aucp$. We recall that $\aucp$ is equivalent to the category $\adfu(\hu_0\opp,\ab)$, where $\hu_0$ is   an essentially wide small subcategory of the retraction-closure of $\hu$ in $\cu$. Moreover,  recall from Remark \ref{rpuresd}(2)  that pure representable functors are the ones represented by elements of $(\cu_{w\le -1}\cup \cu_{w\ge 1})^{\perp}$. Combining these statements with Theorem 4.5.2(II.2) of \cite{bws}  (along with Lemma \ref{lcg} above) we easily obtain the equivalence in question (cf. Proposition \ref{padj} below); one an also deduce it from Proposition\ref{ppcoprws}(\ref{icopr5b}). 

\ref{itnpr}. According to  Lemma 4.4.5 of \cite{neebook}, the subcategory of compact objects of $\cu$ equals $\lan \cp \ra$. Thus it remains to apply Proposition \ref{pexw}(1). 
\end{proof}

\begin{rema}\label{rgensmash}
\begin{enumerate}
\item\label{irspure}
Let us make 
 two simple observations related to pure functors in this setting. 

If $H': \cu\to \au$ is a homological functor that respects coproducts then applying part \ref{itnp2} of our theorem to the opposite (cohomological) functor $H$ from $\cu$ into $\au\opp$ we obtain that $H'$ is pure if and only if it kills $\cup_{i\neq 0}\cp[i]$.

Next,  the description of $\hw$ (combined with Theorem \ref{tpure}) immediately implies that two pure (co)homological functors $H'_1$ and $H'_2$ from $\cu$ that  respect coproducts (resp. $H_1$ and $H_2$ that  convert coproducts into products)  are isomorphic if and only if their restrictions to $\hu$ are.

\item\label{irs3}
Let us now describe a simple example to our theorem that will be useful for us below. 

So, let $\bu$ be a small additive category; define $\hu'$ as the  formal coproductive hull of $\obj \bu$ (see Definition \ref{dbh}(2)).  Then we take  $\cu\subset K(\hu')$ to be the localizing subcategory  that is generated by $\obj \hu'$. Obviously, the set $\cp=\obj \bu$ compactly generates $\cu$, and $\bu$ is connective in $\cu$. Thus there is a unique smashing weight structure $w$ on $\cu$ whose heart contains $\bu$. It is easily seen that this $w$ is the restriction of the weight structure $\wstu$ (see Remark \ref{rstws}(1)) from $K(\hu')$ to $\cu$; see Remark 2.3.2(1) of \cite{bsnew}. Thus $\cu_{w\ge 0}=K(\hu')_{\wstu\ge 0}$, and  part \ref{itnp1d} of our theorem tells us that this class also equals $(\cup_{i<0}\cp[i])^{\perp}$.

\item\label{irs1}
In \S2.3 of \cite{bsnew} (cf. also \S2.2 of ibid.)  a much  more general setting of {\it class-generated} weight structures was considered. In particular, it is not necessary to assume that $\cp$ is a set to have parts \ref{itnp1}--\ref{itnp2}   of our theorem.
So the main problem with the corresponding weight structures is that it is not  clear whether the categories of pure representable functors are "nice enough". 

Another generalization of our theorem was obtained in \S3.3 of \cite{bvtr}.

\item\label{irs5} Furthermore, by Theorem 5 of \cite{paucomp}, for {\bf any} set $\cp$ of compact objects there exists a weight structure with $\cu_{w\ge 0}=(\cup_{i<0}\cp[i])^{\perp}$. 
However, it does not follow that $w$ restricts to the subcategory of compact objects of $\cu$ (as it does in our theorem and in Theorem 3.3.1 of \cite{bvtr}). Respectively, one may call weight structures studied in our theorem {\it purely compactly generated} ones (to distinguish them from general compactly generated ones).

\item\label{irst} As we recall in  Proposition \ref{padj} below, the category  $\hrt$ is the heart of a $t$-structure $t$ that is {\it right adjacent} to $w$ in the sense of \cite[\S1.3]{bvt}; whence the notation.
\end{enumerate}
\end{rema}

Now let us study  pure functors and detecting weights for weight structures provided by our theorem (cf. Remark \ref{rpgcomp} below). 

\begin{pr}\label{pgcomp}
Adopt the notation and the assumptions of Theorem \ref{thegcomp}. 
Let us choose an injective cogenerator $I$ of the category $\hrt$;  we will write $H_I$ for the representable functor $\cu(-,I)$, and $\cacp$ for the Yoneda embedding functor $Q\mapsto (P\mapsto \cu(P,Q)):\ \hw\to \aucp$ (where $P$ runs through $\cp$). Let $M$ be an object of $\cu$, $n\in \z$. 

Then the following conditions are equivalent.

(i). $t(M)\in K(\hw)_{\wstu\ge n}$ (cf. Remark \ref{rwc}(\ref{irwco})).

(ii). $H_j^{\cacp}(M)=0$ for $j<n$.

(iii). $M\perp \cup_{j<n}\{I[j]\}$.

(iv). $H_j(M)=0$ for any $j<n$ and any pure homological functor $H$.
\end{pr}
\begin{proof}
Applying  Proposition \ref{ppcoprws}(\ref{icopr7}) we obtain that $t(M)$  is an object of the localizing subcategory $K'(\hw)$ of $K(\hw)$ generated by $\obj \hw$. Recalling Remark \ref{rgensmash}(\ref{irs3}) we deduce that $t(M)\in K(\hw)_{\wstu\ge 0}$ if and only if $N[j]\perp M$ for all $j<0$ and $N\in \obj \hw$.  Combining this statement with Theorem \ref{thegcomp}(\ref{itnp1d}) we conclude that  conditions (i) and (ii) are equivalent.

 Next, for any $J\in \obj \hrt$ and $j\in \z$ the group $\cu(M[-j],J)$ is isomorphic to the $-j$th cohomology of the complex $\cu(M^s,J)$ according to Theorem \ref{tpure}(2). If $J$ is an injective object of $\hrt$ then this groups is isomorphic to $\aucp(H_j^{\cacp}(M), J)$. Since $I$ cogenerates $\hrt$, we obtain that  $H_j^{\cacp}(M)=0$ if and only if $M\perp I[j]$; hence conditions (ii) and (iii) are equivalent.

Lastly, condition (iv) clearly implies condition (ii), and it follows from condition (i) according to Theorem \ref{tpure}(2).
\end{proof}

\begin{rema}\label{rpgcomp}
 1. This statement can be thought about as a complement to Proposition \ref{pwt}(\ref{iwcons}) (in our setting). Clearly, if $M\in \cu_{w\ge n}$ then it fulfils (all) the conditions of our proposition, and the converse application is fulfilled whenever $M$ is $w$-bounded above.

2. A more or less explicit description of all objects fulfilling these conditions  immediately follows from Corollary 4.1.4(1) 
of \cite{bkwn}. \end{rema}

Let us now prove that the Chow weight structure on $\dmr$ is ("very often") degenerate.

\begin{pr}\label{paydegen}
Let $k$ be a perfect  field, $p=\cha k$; let $R$ be a commutative associative $\zop$-algebra  (where we set $\zop=\z$ if $p=0$). Denote by $\dmr$ the smashing category of $R$-linear Voevodsky motives over $k$  (cf. Remark \ref{rayoub}(1)) 

Then the subcategory $\hu\subset \dmr$ of twists of $R$-motives of smooth projective $k$-varieties by $-(i)[2i]=-\lan i \ra$ is connective and compactly generates $\dmr$ (cf. Theorem \ref{thegcomp}); consequently, $\hu$ purely compactly generates a left non-degenerate weight structure $\wchow$ on $\dmr$.

Moreover, if  $R$ is not torsion and $k$ is of infinite transcendence degree over its prime subfield then 
 $\wchow$ is right degenerate.
\end{pr}
\begin{proof}
These properties of $\hu$ are rather well-known.
The objects of $\hu$ are compact in $\dmr$ essentially by the definition of the latter category; see Remark 4.11 of \cite{degmod}, \S1.1 of \cite{cdint}, or \S1.3 and 3.1 of \cite{bokum}.  In Remark 3.1.2 of \cite{bokum} it is shown that $\hu$ is 
connective in $\dmr$; this is an easy application of Corollary 6.7.3 of \cite{bev}. Lastly, $\dmr$ is well-known to coincide with its localizing subcategory generated by the subcategory $\dmgr$ of compact objects; see \S3.1 of \cite{bokum}. Hence it suffices to recall that $\dmgr$ is densely generated by $\hu$; see Proposition 2.2.1(1) of \cite{bsoscwhn} or  Theorem 2.2.1 of \cite{bzp}. Hence Theorem \ref{thegcomp}(\ref{itnp1}) gives the existence and the left non-degeneracy of  $\wchow$.

Now we demonstrate that 
$\dmr$ contains a right $\wchow$-degenerate object whenever $R$ is not torsion and $k$ is of infinite transcendence degree over its prime subfield. In this case Lemma 2.4 of \cite{ayconj} essentially gives a sequence of morphisms between objects $Y_i=R(i)[i]$ (for $i\ge 0$) such that the motif $Y=\hcl Y_i$ (see Proposition \ref{ppcoprws}(\ref{icoprhcl})) is not zero (note that in loc. cit. formally only the case $R=\q$ and $k\subset \com$ was considered; yet the simple proof of that lemma extends to our setting without any difficulty). Now,  $R(i)[i]\in \obj \chowr[-i]\subset \dmr{}_{\wchow\le -i}$; hence Proposition \ref{ppcoprws}(\ref{icoprhcl}) gives the statement in question.
\end{proof}

\begin{rema}\label{rwchow}
1. Clearly, this weight structure $\wchow$ is also compactly purely generated by any small essentially wide subcategory of $\chowr=\kar(\hu)$ of $R$-Chow motives (whereas the category $\chowr$ itself is essentially small); whence we call it a Chow weight structure (as well). 

2. Moreover, if $R$ is torsion but not zero then one can probably use a similar argument if $k$ is the perfect closure of a purely transcendental extension $k'$ of infinite degree  of some field $k''\subset k'$.\end{rema}

\subsection{On "variations" of weight Postnikov towers and weight complexes}\label{subtlety}

Now we study the question of  lifting $\wstu$-truncations and Postnikov towers of $t(M)$  to  that  of $M$; so we (try to) describe all possible choices of weight complexes for a fixed object of $\cu$ (see Corollary \ref{cenvel}(1)). These questions are  rather natural and actual even  though somewhat technical. 

As a consequence we obtain a new existence of weight structures statement.

\begin{pr}\label{piwcex}
Let $M\in \obj \cu$, $m\in \z$; assume that $\hw$ is Karoubian.

1. Then for a complex $C\in \obj K(\hw)$ there exists a choice of $w_{\le m}M$ such that $t(w_{\le m}M)\cong C$ if and only if $C$ is a choice of $\wstu_{\le m}(t(M))$ (note that this assumption  does not depend on the choice of $t(M)$; 
 see Remark \ref{rwc}(\ref{irwco})).

2. Assume that $N'=w_{\le m+1}M$. Then for any  $m$-$\wstu$-decomposition triangle
\begin{equation}\label{ekwd}
  C \stackrel{e}{\to} t(N')\stackrel{f}{\to} D\stackrel{g}{\to} C[1] \end{equation} 	(in $K(\hw))$), 
where $t(N')$ comes from a  weight Postnikov tower $Po_{N'}$ for $N'$, 	there exist 
certain $\pwcu$-morphisms  $Po_{N'}\stackrel{f_{\pwcu}}{\longrightarrow} Po_{\tilde D}\stackrel{g_{\pwcu}}{\longrightarrow} Po_{N[1]}$   satisfying the following conditions: the corresponding weight complex morphisms  $(f_{K(\hw)}, g_{K(\hw)})$ form a couple $K(\hw)$-isomorphic to $(f,g)$ (with the first "component" of this isomorphism being just $\id_{t(N')}$),   and the underlying $\cu$-morphisms  $N'\stackrel{f_{\cu}}{\to} {\tilde D}\stackrel{g_{\cu}}{\to} N[1]$ extend to a $\cu$-distinguished triangle such that the composed morphism $N\to N'\to M$ gives an $m$-weight decomposition of $M$.

3. Let $\bu$ be a full additive subcategory of $\hw$. Then the full subcategory $\cu^{\bu}$ of $\cu$ consisting of those $N\in \obj \cu$ such that $t(N)$ is $K(\hw)$-isomorphic to an object of 
 $K(\bu)$, is triangulated. Moreover, $w$ restricts (see Definition \ref{dwso}(\ref{idrest})) to $\cu^{\bu}$, and the heart of this restriction $w^{\bu}$ equals $\kar_{\cu}\bu\cong \kar(\bu)$.
\end{pr}
\begin{proof}
1. Applying   Proposition \ref{pwt}(\ref{iwcex})  to any $m$-weight decomposition of $M$ (see Remark \ref{rstws}(2)) we obtain the existence of a $K(\hw)$-distinguished triangle $t(w_{\le m}M)\to t(M)\to t(w_{\ge m+1}M)\to t(w_{\le m}M)[1]$. Next, $t(w_{\le m}M)\in K(\hw)_{\wstu\le m}$ and $t(w_{\ge m+1}M)\in K(\hw)_{\wstu\ge m+1}$ according to Proposition \ref{pwt}(\ref{iwcons}); hence $t(w_{\le m}M)$ is a choice of $\wstu_{\le m}(t(M))$.

	Now we prove the converse implication, i.e., that for any $C=\wstu_{\le m}(t(M))$ there exists $N=w_{\le m}M$ such that $t(N)\cong C$. We take $N'$ to be a choice of $w_{\le m+1}M$; then the complex $C'=t(N')$ (as we have just proved)  is a choice of $\wstu_{\le m+1}(t(M))$. So we   take $f\in K(\hw)(C,C')$ to be 
	 a morphism "compatible with $\id_{t(M)}$" via Proposition \ref{pbw}(\ref{icompl}).
	
	Applying Proposition \ref{pwt}(\ref{iwpt1}) 	 we obtain that $f$ can be completed to a distinguished triangle of the form (\ref{ekwd}). 	Thus the implication in question reduces to assertion 2.

2.	We can certainly assume  $m=-1$. The idea is to "truncate" $N'$ to obtain $N$.

		The aforementioned Proposition \ref{pwt}(\ref{iwpt1}) actually implies that 	$D\in K(\hw)_{\wstu=0}$. Since $\hw$ is Karoubian, we can assume $D\in \obj \hw$; so we set ${\tilde D}=D$. Let us make a choice of $N'{}^0=w_{\ge 0}N'$; it belongs to $\cu_{w=0}$ according to 	Proposition \ref{pbw}(\ref{iwd0}) and arguing is above we obtain that $N'{}^0$ is also a choice of 	$\wstu_{\ge 0} C'$. 	 According to Proposition \ref{pbw}(\ref{ifact}) the morphism $f $ factors through the weight truncation morphism $N'\to N'{}^0$. So we set 		$f_{\cu}$ to be the corresponding composition $N'\to N'{}^0\to D$ and complete this morphism to a distinguished triangle $N 
		{\to}  N'\stackrel{f_{\cu}}{\to} 
		D \stackrel{g_{\cu}}{\to} N[1]$.
	Next we 	apply Proposition \ref{pwt}(\ref{iwcex}) to the couple $(f_{\cu},g_{\cu})$ to 	obtain 	 a choice of morphisms 	$Po_{N'}\stackrel{f_{\pwcu}}{\longrightarrow} Po_{\tilde D}\stackrel{g_{\pwcu}}{\longrightarrow} Po_{N[1]}$   such that $Po_{\tilde D}$ is a weight Postnikov tower for $D$. 
	
				Proposition \ref{pwt}(\ref{iwcex}) also says that the corresponding couple $(f_{K(\hw)}, g_{K(\hw)})$ can be completed to a distinguished triangle. Since $D\in K(\hw)_{\wstu=0}$, the arrow $f_{K(\hw)}$ is $K(\hw)$-isomorphic to  $f$ (since non-zero morphisms in $K(\hw)(N,D)$ do not vanish in $\kw(\hw)(N,D)$). 
		Thus $t(N)\cong C$ and $(f_{K(\hw)}, g_{K(\hw)})\cong (f,g)$. Applying Proposition \ref{pwt}(\ref{iwcons}) we obtain that $N\in\cu_{w\le -1}$.\footnote{Here we use the fact that $N$ belongs to $ \cu_{w\le 0}$ that is immediate from Proposition \ref{pbw}(\ref{iext}).}		Lastly, for a cone $P$ of the composed morphism $N\to N'\to M$ we have a distinguished triangle $D\to P\to w_{\ge 1}M\to D[1]$ (by the octahedral axiom); hence $P\in \cu_{w\ge 0}$, and  therefore $N=w_{\le -1}M$.

3.  Proposition \ref{pwt}(\ref{iwhecat},\ref{iwcex})  clearly implies that $\cu^{\bu}$ is a  triangulated subcategory of $\cu$. Next, to prove that $w$ restricts to  $\cu^{\bu}$ it obviously suffices to verify that for any object $M$ of $\cu^{\bu}$ there exists a choice of $w_{\le 0}M$ that belongs to $\obj \cu^{\bu}$ as well. The latter statement follows immediately from the definition of $ \cu^{\bu}$ along with  assertion 1. 

Let us now  calculate $\hw^{\bu}$. If $M$ is the image of an idempotent endomorphism $p:B\to B$ for $B\in \obj \bu$ then we clearly can take $t(M)=\dots\to 0\to B\stackrel{\id_B-p}{\to} B\stackrel{p}{\to} B \stackrel{\id_B-p}{\to} B\to \dots$ (the 
first $B$ is in degree $0$). 
  Hence $\hw^{\bu}$ contains $\kar_{\cu}(\bu)\cong \kar(\bu)$.
 
Conversely, let $M$ belong to the heart of  $w^{\bu}$. Then we choose $t(M)\in \obj K(\bu)$, and  it  remains to recall that $M$ is a retract of the object $M^0$ according to Proposition \ref{pwt}(\ref{iwch}). 
\end{proof}

\begin{rema}\label{riwcex}
Adopt the assumptions of Proposition \ref{piwcex}(3).

1.  Let us describe a funny example to  this statement.  According to  Proposition \ref{paydegen}  there exists a left non-degenerate   weight structure $\wchow$ on 
  $\dmr$ whose heart is the coproductive hull of $\chowr$ in this category (see Remark \ref{rwchow}(1) and Theorem \ref{thegcomp}(\ref{itnp1},\ref{itnp1d})).  Thus we can take $\bu=\chowr$ in our proposition and consider the corresponding subcategory $\dmr^{\chowr}\subset \dmr$ consisting of motives whose weight complexes are complexes of Chow motives 
	 (and not of retracts of their small coproducts). According to part 3 of our proposition,  this "big" $\wchow$ on $\dmr$ restricts to a weight structure on $\dmr^{\chowr}$  whose heart is equivalent just to $\chowr$.

Next we consider  the functor $F':-\otimes \omp:\da\to \cu'$ (as mentioned in Remark \ref{rayoub}(2)). According to  Theorem \ref{thegcomp}(\ref{itnp1},\ref{itnwe})  there also exists a weight structure on $\cu'$ such that  $F'$ is weight-exact (here we use the negativity statement mentioned in Remark \ref{rayoub}(2)). Thus one may apply Theorem \ref{twcons}  to  the restriction  $F$ of $F'$ to $\da^{\chowk}$ (so, we take $R=k$, where $k$ is our characteristic $0$ base field) as well; note that $\da^{\chowk}$ is obviously bigger than $\dmgk$ (look at $\coprod_{i\in \z}N[i]$ for any non-zero $N\in \obj \chowk$).

 Moreover, it appears (cf. Remark \ref{rayoub}(3)) that combining Theorem II of  \cite{ayoubcon}  with Theorem \ref{twcons}(3) one can 
 deduce that all $\wchow$-bounded below objects of $\da^{\chowk}$ whose de Rham cohomology is zero are right $\wchow$-degenerate.

2. Actually, we would have lost no information if we had assumed that $\bu$ is retraction-closed in $\hw$ in Proposition \ref{riwcex}(3).

Indeed, let us demonstrate that $\cu^{\bu}=\cu^{\kar_{\cu}\bu}$. The latter assertion can be deduced from Remark 1.12(4) of \cite{neederex}; it is obviously equivalent to the 
 fact that any object $C$ of $K(\kar_{\cu}\bu)$ is isomorphic to an object of its full subcategory  $K(\bu)$.
Since $K(\bu)$ is a full triangulated subcategory of $K(\kar_{\cu}\bu)$,  it suffices to consider the case where $C$ is bounded either above or below. Moreover, applying duality 
 we reduce the statement to the case where $C=(C^i)$ is concentrated in non-negative degrees. Now we present each $C^i$ as the image of an endomorphism $p^i$ of some $B^i\in \obj \du$ and consider the morphisms $e^{i}:B^{i}\to B^{i+1}$ that factor through the boundaries $d^{i}:C^{i}\to C^{i+1}$ (cf. the definition of $\kar(B)$ in \S\ref{snotata}). Then it is easily seen that the totalization of the double complex 
$$\begin{CD}
 B^0@>{\id_{B^0}-p^0}>> B^0 @>{p^0}>> B^0 @>{\id_{B^0}-p^0}>> B^0 @>{p^0}>> \dots\\
@VV{e^0}V@VV{0}V@VV{0}V @VV{0}V \\
B^1@>{\id_{B^1}-p^1}>> B^1 @>{p^1}>> B^1 @>{\id_{B^1}-p^1}>> B^1 @>{p^1}>> \dots\\
@VV{e^1}V@VV{0}V@VV{0}V @VV{0}V \\
B^2@>{\id_{B^2}-p^2}>> B^2 @>{p^2}>> B^2 @>{\id_{B^2}-p^2}>> B^2 @>{-p^2}>> \dots\\
@VV{e^2}V@VV{0}V@VV{0}V @VV{0}V \\
\dots@>{}>>\dots @>{}>>\dots @>{}>>\dots
\end{CD}$$
is homotopy equivalent to $C$ (in $K(\kar_{\cu}\bu)$).

3. Now let us assume that $\bu$ is retraction-closed in $\hw$ (though this condition is not really important, as we have just demonstrated),  and suppose that $\cu'$ is a full triangulated subcategory of $\cu$ such that $w$ restricts to it and the heart of this restriction lies in  $\bu$. Then any object of $\cu'$ clearly possesses a $w$-weight complex
whose terms are objects of $\bu$. Thus $\cu'$ is a subcategory of $\cu^{\bu}$, and we obtain that $\cu^{\bu}$ is the maximal subcategory of $\cu$ such that $w$ restricts to it and the heart of this restriction is $\bu$. 

Note in contrast that $w$ restricts to the subcategory of $\cu$  strongly generated by $\bu$ and the heart of this restriction equals $\bu$ according to Proposition \ref{pexw}(1) (combined with Proposition \ref{pbw}(\ref{igenlm})); this restriction of $w$ is essentially  minimal  among the restrictions whose heart equals $\bu$.

\end{rema}

Now we prove a corollary that is actual for \cite{bsoscwhn}.

\begin{coro}\label{cenvel}
Assume that $M\in \obj \cu$ and $t(M)$ is homotopy equivalent to a complex $(M'{}^i)\in \obj K(\hw)$.

1. Assume that $M\in \cu_{w\le 0}$ and $M'{}^i=0$ for $i<0$. Then 
there exists a weight Postnikov tower for $M$ such that the corresponding weight complex $t'(M)$ equals $(M'{}^i)$.

2. Assume that $M$ is $w$-bounded and  the complex $(M'{}^i)$ is bounded. Then $M$ belongs to the $\cu$-extension closure of  $M'{}^i[-i]$.
\end{coro}
\begin{proof}
1. We set $M_{\le i}=M$ for $i\ge 0$. According to Theorem 2.2.2(I.1,III.1) of \cite{bonspkar}, there exists a triangulated category $\cu'$ endowed with a weight structure $w'$ such that $\cu\subset \cu'$, $w$ is the restriction of $w'$ to $\cu$, and that the category $\hw'$ is Karoubian. Moreover, we can certainly assume that $\cu$ is a strict subcategory of $\cu'$.

 Now we apply Proposition \ref{piwcex}(2) to $\cu'$ repetitively starting from $M_{\le 0}=M$ to obtain a $w'$-Postnikov tower for $M$ such that $M^i=M'{}^i$ for all $i\in \z$. Being more precise, we set  $M_{\le i}=M$ for $i\ge 0$, choose a  
 $w'$-Postnikov tower for $M=M_{\le 0}$, 
 and starting from $i=0$ on each step we take a $w'$-Postnikov tower $Po_{M_{\le i}}$ constructed earlier and construct a distinguished triangle $M_{\le i}\stackrel{b_i}{\to} M'{}^{-i}[i]\stackrel{c_i}{\to} M_{\le i-1}[1]\to M_{\le i}[1]$ along with a lift of $(b_i,c_i)$ to  $\operatorname{Post}_{w'}(\cu')$; we demand $(b_{i,K(\hw)}, c_{i,K(\hw)})$ to be isomorphic to the corresponding morphisms in the tower corresponding to the complex $(M'{}^i)$. These compatibilities of choices (of  $w'$-Postnikov towers) allow us to  compute $t(d'^{-i}_{M}): M'{}^{-i}\to  M'{}^{1-i}$ as  $b_{i,K(\hw)}[1-i]\circ c_{i,K(\hw)}[-i]$. Since the restriction of $t$ to the lifts to $\cuw$ of $\hw$ is a full embedding, we obtain that the boundary morphisms in this $t'(M)$ as the same as that for $M'{}^i$ if $i\ge 0$, and clearly $M'{}^i=0$ for $i<0$. 

It remains to check that this $w'$-Postnikov tower for $M$ is simultaneously a $w$-one. Since $w$ is the restriction of $w'$ to $\cu$ and $M'{}^i\in \obj \cu$, for this purpose it suffices to verify that all $M_{\le i}$ are objects of $\cu$ as well; the latter fact is given by an  obvious downward induction (if $i<0$; whereas for $i\ge 0$ we have $M_{\le i}=M\in \obj \cu$). 

2. We can certainly assume that $M'{}^i=0$ for $i<0$ and $M\in \cu_{w\ge 0}$.

Let us take some weight Postnikov tower for $M$ such that the corresponding complex $t'(M)$ equals $M'{}^i$ (as provided by the previous assertion). Assume that $M'{}^i=0$ for $i\ge j$. Then the corresponding weight complex $t(M_{\le -j})$ is zero, and Proposition \ref{pwt}(\ref{iwcons}) implies that $M_{\le -j}=0$. Hence this Postnikov tower is bounded (see Definition \ref{dfilt}(1)); thus $M$ belongs to the $\cu$-extension closure of  $M'{}^i[-i]$ by 
 Lemma \ref{lrwcomp}(3). 
\end{proof}

\begin{rema}\label{realiz}
1. The author suspects that a more detailed treatment of weight Postnikov towers would give a complete description of those (not necessarily bounded from any side) $(M'{}^i)\in \obj K(\hw)$ that come from  $w$-Postnikov towers for $M$. Most probably, any $(M'{}^i)$ that is $K(\hw)$-isomorphic to $t(M)$ can be "realized" this way if $\hw$ is Karoubian.

It appears that this statement is easier to prove whenever there exists a "lift" of $t$ to an exact functor  $t^{st}:\cu\to K(\hw)$
(cf. Remark \ref{rwc}(\ref{irwc7})). Indeed, then one can lift the stupid truncations of $(M'{}^i)$ via $t^{st}$ to weight truncations $w_{\le j}M$ (using the corresponding minor modification of Proposition \ref{piwcex}(2));  connecting these objects via Proposition \ref{pbw}(\ref{icompl}) gives a weight Postnikov tower as desired.

2. Note however that  it is important to assume that $\hw$ is Karoubian. Indeed, if $p$ is an idempotent endomorphism of $B\in \obj \hw$ that does not give an $\hw$-splitting then the $2$-periodic complex $\dots \to B\stackrel{\id_B-p}{\longrightarrow} B\stackrel{p}{\to} B \stackrel{\id_B-p}{\longrightarrow} B\to \dots$ does not come from a weight Postnikov tower for $M$ since the corresponding $M_{\le 0}$ is a retract of $B$ that cannot belong to $\obj \cu$.

This example demonstrates that the "lifting" questions treated in this section are not   trivial.
\end{rema}

\section{"Topological" examples}\label{stop}

In \S\ref{shg} we prove that in the equivariant stable homotopy category $\shg$ of $G$-spectra there exists a weight structure $\wg$ generated by the stable orbit subcategory (that consists of the spectra of the form $S_H^0$, 
 where $H$ is a closed subgroup of $G$, i.e., $\wg$ is purely compactly generated by equivariant spheres). Applying the previous results to $\wg$ we obtain that this weight structure is closely related to the connectivity of spectra,  and the corresponding pure cohomology is  the Bredon one  (coming from Mackey functors). 

Our section \ref{shtop} is dedicated to the detailed study of the case $G=\{e\}$, i.e., of the stable homotopy category $\shtop$ (along with the corresponding weight structure $\wsp$). We prove that $\wsp$-Postnikov towers are the {\it cellular} ones in the sense of \cite{marg} (cf. Remark \ref{rshg}(\ref{irgcell}) for a conjectural generalization of this statement). 

In \S\ref{spost} we discuss the relation of our results to $t$-structures (that are {\it adjacent} to the corresponding weight structures) along with certain results of \cite{axstab}; we do not prove anything new in this section.

\subsection{On equivariant spherical weight structures and Mackey functors}\label{shg}

Now let us apply the general theory to the study of equivariant stable homotopy categories. Let us list some notation and definitions related to this matter.

\begin{itemize}
\item $G$ 
 is a (fixed) compact Lie group; we will write $\shg$ for the stable homotopy category of $G$-spectra indexed by a complete $G$-universe.

\item We take $\cp$ to be the set of spectra of the form $S_H^0$, where $H$ is a closed subgroup of $G$ (cf. Definition I.4.3 of \cite{lms}; recall that $S_H^0$ is constructed starting from the $G$-space $G/H$). We will use the notation $\hu$ for the corresponding (preadditive) subcategory of $\shg$; so, it is the (stable) {\it  orbit category} of ibid. 

Recall also that $S_H^0[n]$ is the corresponding sphere spectrum $S_H^n$ essentially by definition (see loc. cit.). 

\item The equivariant homotopy groups of an object $E$ of $\shg$ are defined as $\pi_n^H(E)=\shg(S_H^n,E)$ (for all $n\in \z$; see \S I.6  and Definition I.4.4(i) of ibid.). 

\item We will write $\emg$ for the full subcategory of $\shg$ whose object class is $(\cup_{i\in \z\setminus \ns} \cp[i])^{\perp}$ (its objects are the {\it Eilenberg-MacLane} $G$-spectra; see \S XIII.4 of \cite{mayeq}).

\item We 
 write $\macg$ for the category of   additive contravariant functors from $\hu$ into $\ab$ (cf. Theorem \ref{thegcomp}(\ref{itnp3})); its objects are the {\it Mackey functors} in the sense of loc. cit. Respectively, $\cacp$ will denote the Yoneda embedding $\hu\to \macg$. 

Recall here that for any Mackey functor $M$ the corresponding (Bredon cohomology) functor $H_G^0(-,M)$ is represented by some Eilenberg-MacLane $G$-spectrum (see loc. cit.).
\end{itemize}

Now we  relate the theory of weight structures (cf. Definitions \ref{dwstr}, \ref{dwso}(\ref{idh},\ref{idrest}), \ref{dpure}, and \ref{dbh}) to $\shg$. 

\begin{theo}\label{tshg}
 The following statements are valid.

\begin{enumerate}
\item\label{itgfine}
The category $\cu=\shg$ and the class  $\cp$ specified above  satisfy the assumptions of Theorem \ref{thegcomp}. Thus $\cp$ gives a a left non-degenerate smashing weight structure $\wg$ on $\shg$ whose heart  is the coproductive hull of $\hu$ in $\shg$ and is equivalent to the  formal coproductive hull of $\hu$. 

\item\label{itgrest} $\wg$ restricts to the subcategory of compact objects of $\shg$; the heart of this restriction $\wgfin$ is the Karoubi envelope of the class of finite coproducts of elements of $\cp$.

\item\label{itgconn}
Let $n\in \z$. Then the class of $n-1$-connected spectra (see Definition I.4.4(iii) of \cite{lms}; i.e., this is the class $(\cup_{i<n}\cp[i])\perpp$)) coincides with $\shg_{\wg\ge n}$. In particular, $\shg_{\wg\ge 0}$ is the class of {\it connective} spectra, and $\wg$-bounded below objects are the  bounded below spectra of loc. cit.

Moreover, $\shg_{\wg\ge n}$ is the smallest class of objects of $\shg$ that contains $\cp[i]$ for $i\ge n$ and is closed with respect to coproducts and extensions.

\item\label{itgpure} A (co)homological functor from $\shg$ into an  AB4 (resp. AB4*)  abelian category $\au$ that respects coproducts (resp. converts them into products) is $\wg$-pure if and only if it kills $\cup_{i\neq 0}\cp[i]$. 

In particular, all Eilenberg-MacLane $G$-spectra represent $\wg$-pure functors.

\item\label{itgcacp} The pure homological functor $H^{\cacp}$ (see Theorem \ref{tpure} and Proposition \ref{pgcomp}) is the 
equivariant   ordinary homology with Burnside ring coefficients   functor $H^G_0$ considered in \cite{lewishur} (cf. also Definition X.4.1 of \cite{mayeq}),  and for any Mackey functor $M$ the corresponding pure functor $H_M$ (see Remark \ref{rpuresd}(1)) coincides with $H_G^0(-,M)$  in Definition X.4.2 and \S XIII.4 of ibid. 

\item\label{itgemg} The category $\emg$ is naturally isomorphic to $\macg$ (in the obvious way); thus $\emg$ is Grothendieck abelian and has an injective cogenerator $I$.

\item\label{itgneg} $ \shg_{\wg\le 0}$  is the smallest subclass of $\obj \shg$ that is closed with respect to coproducts, extensions, and contains $\cp[i]$ for $i\le 0$. This class also equals  $   \perpp \shg_{\wg\ge 1}$;\footnote{Recall here $\shg_{\wg\ge 1}$ is the class of $0$-connected $G$-spectra.} 
 moreover, it is annihilated by $H_i$ for all $i>0$ and 
 every pure homological functor $H$ from $\shg$.
\end{enumerate}
\end{theo}
\begin{proof}
\ref{itgfine}. The compactness of the spectrum $S_H^0$  for any closed subgroup $H$ of $G$ is given by Corollary A.3 of \cite{hukriz}; cf. also Lemma I.5.3 of \cite{lms}. The category $\hu$ is a connective subcategory of $\shg$  according to Lemma 2.3(i) of \cite{lewishur} (see also Proposition I.7.14 of \cite{lms} for a generalization). 

Next, $\cp$ generates $\shg$ as its own localizing subcategory according to Theorem 9.4.3 of \cite{axstab}.\footnote{Alternatively, note that the definition of $\shg$ (see \S I.5 of of \cite{lms}) immediately implies that $(\cup_{i\in \z}\cp[i])\perpp=\ns$, and it remains to apply Lemma \ref{lcg} to obtain this generation statement.}  

Lastly, we apply Theorem \ref{thegcomp}(\ref{itnp1d}) to compute the heart of $\wg$.

\ref{itgrest}. This is just the corresponding case of Theorem \ref{thegcomp}(\ref{itnpr}).

\ref{itgconn}. By definition, a $G$-spectrum $N$ is $n-1$-connected whenever $\pi_i^H(N)\cong \shg(S_H^i,N)= \ns$ for any $i<n$ and  any closed subgroup $H$ of $G$. Hence it remains to apply Theorem \ref{thegcomp}(\ref{itnp1d}) to obtain all the statements in question.

\ref{itgpure}. Immediate from Theorem \ref{thegcomp}(\ref{itnp2}) (cf. also Remark \ref{rgensmash}(\ref{irspure})).

\ref{itgcacp}. Since the functor $H^G_0$ (resp. $H_G^0(-,M)$)   respects coproducts (resp. converts coproducts into products), combining the previous assertion with Remark \ref{rgensmash}(\ref{irspure})  we obtain that it suffices to compare the restrictions of $H^G_0$ and $H^{\cacp}$ (resp. $H_G^0(-,M)$ of $H_M$) to the categories $\hu[i]$ for $i\in \z$. Now, 
 these restrictions are canonically isomorphic by definition (resp. by the "dimension axiom" (XIII.4.4) of \cite{mayeq}).

\ref{itgemg}. This is just the corresponding case of Theorem \ref{thegcomp}(\ref{itnp3}).

\ref{itgneg}. The first of these descriptions of $ \shg_{\wg\le 0}$ is given by Theorem \ref{thegcomp}(\ref{itnpb}). Next,
  $ \shg_{\wg\le 0}=   \perpp \shg_{\wg\ge 1}$ according to Proposition \ref{pbw}(\ref{iort}).  It remains to recall the definition of pure functors	to conclude the proof. 
\end{proof}

\begin{rema}\label{rshg}
\begin{enumerate}
\item\label{irglam}
Our theory  can also be applied to the subcategories of $\lam$-linear spectra  in $\shg$; that is, for a  set of prime numbers $S\subset \z$ and $\lam=\z[S\ob]$ one can invert all morphisms of the form $s\id_E$ for all $s\in S$ and $E\in \obj \shg$, and the right adjoint $l_S$  gives an embedding of this localization $\shg_{\lam}$ into $\shg$. These versions of our statements are described in \S4.2 of \cite{bkwn}.

\item\label{irgcell} It is certainly very convenient to compare pure functors by looking at their restrictions to the subcategory $\hu$ of $\hw$ only; this was crucial for the proof of part \ref{itgcacp} of our theorem. Recall however that the (co)homology functors coming from Mackey functors can be described in terms of {\it skeletal filtrations of $G$-CW-spectra} (see \S XIII.4 of \cite{mayeq}). Thus it would be nice to understand the relation between filtrations of this sort and $\wg$-Postnikov towers. 

The author conjectures that a $\wg$-Postnikov tower of a spectrum $E$ lifts to a skeletal filtration (in the category of $G$-CW-spectra) if and only if the corresponding terms $E^i\in \shg_{\wg=0}$ are coproducts of elements of $\cp$ (i.e., one cannot take retracts of objects of this type here). Note that the proof of this statement would closely relate the filtrations on $\shg$ given by $\wg$ with {\it $(a,b)$-dimensional spectra} of \cite[Appendix A]{green} (one considers the quotients of possible skeletal filtrations for object of $\shg$ in this definition).\footnote{Probably, Corollary \ref{cenvel}  can help to study this relation for {\it finite} $G$-spectra. One of the problems here is that the additive hull of $\cp$ is not Karoubian in general; cf. Theorem A.4 and Remarks A.5 of ibid.} The author believes that this conjecture is rather easy for those spectra that come from {\it $G$-CW-complexes}, and requires certain (countable) homotopy colimit and limit argument (cf.  Definition 2.1 of \cite{bokne} or Proposition \ref{ppcoprws}(\ref{icoprhcl})) in general.

Note however that in the case $G= \{e\}$ there exists an alternative notion of  {\it cellular tower} considered in \S6.3 of \cite{marg}. Theorem \ref{top}(\ref{itopcell}) below  says that this notion is equivalent to that of a $\wsp$-Postnikov tower;  this is related to the fact that in this case the category of coproducts of  copies of $S^0$ (along with its subcategory of finite coproducts) is Karoubian. 

\item\label{iruniv}
It appears that the results of \cite{lms} are actually sufficient to generalize all the assumptions of our theorem  to the case where $\cu$ is the stable homotopy category of $G$-spectra indexed on a not (necessarily) complete $G$-universe. In particular,  this universe can be $G$-trivial (i.e.,  $G$ acts trivially on it); this  allows us to apply it to the corresponding representable functors as considered in \S IV.1 of \cite{bred}. 
\end{enumerate} \end{rema}

Now let us formulate our weight detection statements in this setting. We will freely use the notation introduced above.

\begin{pr}\label{pshgwd}
Consider the following conditions on an object $E$ of $\shg$.

(i). $E$ is connective.

 (ii). $C_j(E)=0$ for any  $\wg$-pure homological functor $C$ from $\shg$ and $j<0$.

(iii). $C^j(E)=0$ for any  $\wg$-pure cohomological functor $C$ from $\shg$ and $j<0$.

(iv). $E\perp I[j]$ (see Theorem \ref{tshg}(\ref{itgemg})) for all $j<0$.

(v). $H_j^{\cacp}(E)=0$ for all $j<0$.

Then the following statements are valid. 

1. Conditions (ii)-(v) are equivalent. 

2. These conditions follow from condition (i).

 3. The converse implication is fulfilled whenever $E$ is $m$-connected for some $m\in \z$.
\end{pr}
\begin{proof}
1. Conditions (ii) and (iii) are equivalent since we can just invert arrows in the target.  It remains to apply  Proposition \ref{pgcomp} to our setting.

2. Condition (i) implies condition (ii) just by the definition of purity.

3. Recall that Proposition \ref{pgcomp} also implies that condition (ii) (as well as (iv) and (v)) are equivalent to $t_{\wg}(E)\in K(\hw^G)_{\wstu \ge 0}$. Thus it remains to apply Proposition \ref{pwt}(\ref{iwcons}).
\end{proof}

\subsection{The case of trivial $G$: cellular towers}\label{shtop}

Now we apply our results to the  stable homotopy category $\shtop$ (whose detailed description can be found in \cite{marg}). This corresponds to the case of a trivial $G$ in Theorem \ref{tshg}. We  will write $\emo$ for $\emg$ and $w^{sph}$ for $\wg$ in this case,  and use the remaining notation from this theorem. 

\begin{theo}\label{top}
Set $\cp=\{S^0\}$. 

 Then the following statements are valid.

\begin{enumerate}

\item\label{itoptriv}
The functor $\shtop(S^0,-)$ gives equivalences $\hwsp\to \abfr$ (the category of free abelian groups) and $\emo\to \ab$; thus $\aucp$ is equivalent to $\ab$ as well.

Moreover, $\wsp$ restricts to the category of finite spectra, and the heart of this restriction is equivalent to the category of  finitely generated  free abelian groups.

\item\label{itopsingh}
The functor $\cacp$ is essentially the singular homology functor $\hsing$.

\item\label{itopsingc}
For any abelian group $\gam$ and the corresponding spectrum $\egam\in \obj \emo$ the functor $\shtop(-,\egam)$ is isomorphic to  the singular cohomology with coefficients in $\gam$. 

\item\label{itopcell}
A  $\wsp$-Postnikov tower of a spectrum $E\in \obj \shtop$ is a cellular tower  for $E$ in the sense of 
 \cite[\S6.3]{marg}.

\item\label{itopskel}
Assume that the following statement is fulfilled for any object $E$ of $\shtop$: if $\hsing_i(E)=\ns$ for $i>0$ and $\hsing_0(E)$ is a free abelian group then $E\in \shtop_{\wsp\le 0}$.

Then (conversely) a cellular tower  for $E$ in the sense of loc. cit. is a $\wsp$-Postnikov tower, and 
$\shtop_{\wsp\le n}$ consists precisely of {\it$n$-skeleta} (of certain spectra) in the sense of ibid. (cf. also Definition 6.7 of \cite{christ}). 
\end{enumerate}

\end{theo}
\begin{proof}
\ref{itoptriv}. Since the endomorphism ring of $S^0$ is $\z$, $\aucp$ is equivalent to $\ab$. Next,  Theorem 
\ref{tshg}(\ref{itgfine}) implies that $\hwsp$ is  equivalent to the Karoubi envelope of the category of free abelian groups, i.e., to   $\abfr$ itself. Moreover, applying 
Theorem \ref{tshg}(\ref{itgrest}) we obtain that $\wsp$ restricts to finite spectra and compute the  heart of this restriction. Lastly, 
 applying part \ref{itgemg} of that theorem we obtain that $\emo$ is equivalent to $\ab$ as well. 


\ref{itopsingh}. Obviously, singular homology respects coproducts (since $\hsing(-)\cong \shtop(S^0,-\wedge \emz)$ and the smash product in $\shtop$ preserves coproducts; here $\emz$ is the Eilenberg-MacLane spectrum corresponding to the group $\z$). It is also pure since $\hsing_i(S^0)=\ns$ for $i\neq 0$. Thus comparing the restrictions of $\hsing$ and $\cacp$ to the category $\hu$ (whose only object is $S^0$ in this case) we obtain the isomorphism in question according to the aforementioned Theorem \ref{thegcomp}(\ref{itgemg}) (cf. also Remark \ref{rgensmash}(\ref{irspure})).

\ref{itopsingc}. This is a well-known statement;  it can also be easily deduced from Theorem \ref{thegcomp}(\ref{itgemg}). 

\ref{itopcell}. Let us recall that a cellular tower for $E$ is a certain Postnikov tower for $E$ whose term $E_{\le n}$ (in the notation of Definition \ref{dfilt}(2)) was denoted by $E^{(n)}$ in \S6.3 of \cite{marg}. Respectively, the assumption on $E_n=\co(E_{\le n-1}\to E_{\le n})$ in loc. cit. says that $E_n$ is a coproduct of $S^0[n]$ (since $S^0[n]$ is the $n$-dimensional sphere spectrum) for all $n\in \z$; this statement is  clearly fulfilled for  $\wsp$-Postnikov towers (see Proposition \ref{pwt}(\ref{iwpt1})). 

Next,  $E$ should be the {\it minimal weak colimit} of  $E_{\le n}$ in the sense of \S3.1 of ibid. 
Now, 
 $\wsp$ is a smashing left non-degenerate weight structure; hence 
 $\hcl_n \wsp_{\le n}E\cong E$ by Theorem 4.1.3(1,2) of \cite{bsnew}; here $\hcl_n \wsp_{\le n}E$ denotes the {\it countable homotopy colimit} of these objects (see  Proposition \ref{ppcoprws}(\ref{icoprhcl})).  It easily follows that the induced homomorphism $\shtop(E,Y)\to \prli_n \shtop(\wsp_{\le n} E ,Y)$ is surjective for any $Y\in \obj \shtop$, and  the stable homotopy groups $\pi_*(E)$ of $E$ are the colimits of those of $\wsp_{\le n}E$  (note that $\pi_i= \shtop(S^0[i],-)$); here one can apply  the well-known Remark 4.1.2(2,3) of \cite{bsnew}. 
 Thus $E$ is the minimal weak colimit of $\wsp_{\le n}E$ by definition. 

Furthermore, assertion \ref{itopsingh} of 
 this  theorem (combined with the definition of pure homological functors)  immediately implies that the inverse limit of  singular homology of $\wsp_{\le n}E$ (when $n$ goes to $-\infty$) vanishes; this is the last of the conditions in the definition of cellular towers.

\ref{itopskel}. Assume that $E^{(n)}=E_{\le n}$ is a cellular tower for $E$ in the sense of Margolis. Applying Proposition 6.12 
  of \cite{marg} we obtain that $\hsing_i(E_{\le m})=\ns$ whenever $i>m\in \z$, and $\hsing_m(E_{\le m})$ is a free abelian group. 	 Applying our assumption we obtain that $E_{\le m}\in \shtop_{\wsp\le m}$.

	Next, since $E_n\in \shtop_{\wsp=n}$, the corresponding distinguished triangles (\ref{etpt}) imply that 	the homomorphism $\pi_i(E_{\le n})\to \pi_i(E_{\le n+1})$ is bijective if $i<n\in \z$ and surjective if $i=n$. Applying the isomorphism $\inli_n  \pi_i(E_{\le n})\cong \pi_i(E)$ we obtain that $\pi_i(\co(E_{\le n}\to E))=\ns$ if $i\le n$. 
	Hence $\co(E_{\le n}\to E)\in \shtop_{w\ge n+1}$ according to Theorem \ref{tshg}(\ref{itgconn}). Thus $E_{\le n}\to E\to \co(E_{\le n}\to E)\to E_{\le n}[1]$ is an $n$-weight decomposition of $E$, and we obtain that $E_{\le n}$ give a weight Postnikov tower indeed.

To prove that second statement in our assertion it remains to recall that $n$-skeleta are just the spectra of the form $E^{(n)}$ for  $E\in \obj \shtop$ (and some cellular tower for $E$). 
\end{proof}

\begin{rema}\label{rmarg}
1. The assumption made in part \ref{itopcell} of our theorem  is given by Theorem 4.2.3(6) of \cite{bkwn}.  
 Moreover, if we assume in addition that  $E$ belongs to $\shtop_{\wsp\le m}$ for some $m\in \z$ then  the  statement easily follows from Proposition \ref{pwt}(\ref{iwcons}),   whereas the general case requires the theory of {\it killing weights} (as developed in ibid.). 

Note also that combining loc. cit. and related results with Theorem \ref{tshg}  one obtains several new statements on $\shg$. 

2. Parts \ref{itoptriv}-- \ref{itopcell} 
 of our theorem were previously formulated by the author in \S4.6 of \cite{bws}. However, their proofs sketched in loc. cit. contained significant gaps,  and the results of the current paper and of \cite {bkwn}  really help in closing them.

Note also that part \ref{itopcell} of our theorem suggests that weight spectral sequences corresponding to $\wsp$ can be called Atiyah-Hirzebruch ones.

3. Since cellular towers are $\wsp$-Postnikov towers, we obtain that Proposition 6.18 of \cite{marg} is the corresponding case of the weak functoriality of weight Postnikov towers (see Proposition \ref{pwt}(\ref{iwpt2})).
\end{rema}

\subsection{
The relation to adjacent $t$-structures and connective stable homotopy theory}\label{spost}

Now we  relate the result of this section to $t$-structures. Though in \cite{bbd} (where this notion was introduced) and in previous papers of the author the "cohomological convention" for $t$-structures was used, we prefer to use the homological convention that is "more coherent" with the topological examples.  The corresponding versions of the definitions are as follows.

\begin{defi}\label{dtstrh}
1. A couple of subclasses  $\cu_{t\le 0},\cu_{t\ge 0}\subset\obj \cu$ will be said to be a
$t$-structure $t$ on $\cu$  if 
they satisfy the following conditions:

(i) $\cu_{t\le 0}$ and $\cu_{t\ge 0}$  are strict, i.e., contain all
objects of $\cu$ isomorphic to their elements.

(ii) $\cu_{t\le 0}\subset \cu_{t\le 0}[1]$ and $\cu_{t\ge 0}[1]\subset \cu_{t\ge 0}$.

(iii)  $\cu_{t\ge 0}[1]\perp \cu_{t\le 0}$.

(iv) For any $M\in\obj \cu$ there exists a  {\it $t$-decomposition} distinguished triangle
\begin{equation}\label{tdec}
L_tM\to M\to R_tM{\to} L_tM[1]
\end{equation} such that $L_tM\in \cu_{t\ge 0}, R_tM\in \cu_{t\le 0}[-1]$.

2. $\hrt$ is the full subcategory of $\cu$ whose object class is $\cu_{t=0}=\cu_{t\le 0}\cap \cu_{t\ge 0}$.
\end{defi}

 \begin{rema}\label{rtst}
Let us recall some well-known properties of $t$-structures (cf. \S1.3 of \cite{bbd}).

1. 
The triangle (\ref{tdec}) is canonically and functorially determined by $M$; thus $L_t$ and $R_t$ are functors. 

2. $\hrt$ is   an abelian category with short exact sequences corresponding to distinguished triangles in $\cu$.

Moreover, 
 $L_t\circ [1] \circ R_t\circ [-1] \cong [1] \circ R_t\circ [-1] \circ L_t$ (if we consider these functors as endofunctors of $\cu$). The corresponding  composite functor $H^t$ actually takes values in $\hrt\subset \cu$, and it is homological if considered this way. 

3. We have $\cu_{t\le 0}=(\cu_{t\ge 0}^{\perp})[1]$. Thus $t$ is uniquely determined by $\cu_{t\ge 0}$ (and  actually by $\cu_{t\le 0}$ as well). 
\end{rema}

Now let recall some relations of $t$-structures to the results and formulations above.

\begin{pr}\label{padj}

Adopt the notation and assumptions of Theorem \ref{thegcomp}. 
Then the following statements are valid.

\begin{enumerate}
\item\label{ipa1}
$w$ is {\it left adjacent}  to a certain (unique) $t$-structure $t$ in the sense of  \cite[Proposition 1.3.3]{bvtr}, i.e., $\cu_{t\ge 0}=\cu_{w\ge 0}$. 

\item\label{ipa2} The category $\hrt\subset \cu$ is the heart of this $t$, and the functor $H^t$ becomes isomorphic to $H^\cacp$ when  combined with the equivalence $\hrt\to \aucp$ provided by Theorem \ref{thegcomp}(\ref{itnp3}).
\end{enumerate}
\end{pr}
\begin{proof}  Immediate from  Theorem 4.5.2 of \cite{bws}; one can also combine Theorem 3.2.3(I) of \cite{bvtr} with Theorem \ref{thegcomp}(\ref{itnp1d}). 
\end{proof}

\begin{rema}\label{radjt}
 1. For $\cu=\shtop$ (and $\cp=\{S^0\}$) the corresponding $t$-structure is often called the Postnikov $t$-structure; so the author suggest to use this term for $\cu=\shg$ (and for $\cp$ as in the Theorem \ref{tshg}) as well.  In \S3.2.B of \cite{marg} the object $L_tE$ (resp. $R_tE$; here $E$ is an object of $\shtop$) in the distinguished triangle (\ref{tdec}) was said to be  {\it of type $E[0,\infty]$} (resp. $E[-\infty,-1]$).
 Respectively,  in (the proof of) Proposition 3.8 of ibid. certain $t$-decompositions (\ref{tdec}) were constructed; see also  Proposition 7.1.2(c) of \cite{axstab} for a more general (and "explicit") formulation of this sort.

Note also that this  $t$ does not restrict to finite spectra in contrast to $\wsp$. 

2. Let us now say more on the relation of weight structures to connective stable homotopy theory as discussed in 
\S7 of \cite{axstab}.

In that section the corresponding triangulated category (we will write $\cu$ for it) was assumed to be {\it monogenic}, i.e., $\cu$ is compactly generated by a single element set $\{X\}$ for some $X\in \obj \cu$. Next, the {\it connectivity} assumption on $X$ in ibid. means that $X\perp X[i]$ for all $i>0$. Applying Theorem \ref{thegcomp}(\ref{itnp1},\ref{itnp1d}) we obtain a unique smashing weight structure on $\cu$ whose heart consists of all retracts of coproducts of copies of $X$.

Now,  our weight structure approach give some new statements on the corresponding class $\cu_{w\ge 0}=\cu_{t\ge 0}$. We also obtain  the class  $\cu_{w\le 0}$ that appears to be new and important (even in the case $\cu=\shtop$). In particular, we use this class to give a  nice definition of weight Postnikov towers (that can also be called cellular towers) for arbitrary objects of   $\cu$ (in contrast to Proposition 7.1.2(a) of ibid.). The notions of weight complex and weight spectral sequence are also new for this context.

Moreover, note that we currently have a good understanding of weight-exact localizations (achieved in \cite[\S4]{bos} and \cite[\S3]{bsnew}; cf. the beginning of \S7 of \cite{axstab}).
\end{rema}

\appendix
\section{Some details on weight complexes} 
 \label{sdwc}

In this section we discuss certain problems of the theory of weight complexes and weight spectral sequences, and of the exposition of this theory in \cite{bws}.
 So we give the proof of  Proposition \ref{pwt}(\ref{iwpt1}--\ref{iwpt3},\ref{iwhefun},\ref{iwcex},\ref{irwcons},\ref{iwc2342})  in \S\ref{sirwc};   the reason for doing this is described in Remark  \ref{rwcbws}(\ref{i324},\ref{iiwcex}).  In  \S\ref{s3bws} we also discuss the problems caused by the non-self-duality of the construction of weight complexes and weight spectral sequences. 

\subsection{The proof of Proposition \ref{pwt}(\ref{iwpt1}--\ref{iwpt3},\ref{iwhefun},\ref{irwcons},\ref{iwcex},\ref{iwc2342})}\label{sirwc}
\begin{proof}

We will freely use the notation from Definition \ref{dfilt}.

\ref{iwpt1}--\ref{iwpt3}. First we verify the second statement in   Proposition \ref{pwt}(\ref{iwpt1}), i.e., that $M^i$ belong to $ \cu_{w=0}$ for all $i\in \z$. 
 Obviously, it suffices to prove this statement for $i=0$.

By definition, $M^0=M_0$ is an extension of  $w_{\le 0}M$ by $(w_{\le -1}M)[-1]$. Thus $M_0=M^0$ belongs to $ \cu_{w\le 0}$ by Proposition \ref{pbw}(\ref{iext}).

Similarly we obtain the following: to prove that $M_0$ belongs to  $ \cu_{w= 0}=\cu_{w\ge 0}\cap   \cu_{w\le 0}$ as well, it remains to  verify that $M_0$ is an extension of $w_{\ge 0}M$ by $(w_{\ge 1}M)[-1]$.
To prove the latter fact we complete the commutative triangle $w_{\le -1}M \stackrel{j_{-1}}{\longrightarrow} w_{\le 0}M \stackrel{h_0}{\longrightarrow} M$ to an octahedron
\begin{equation}
\label{doct} \xymatrix{ w_{\ge 1}M \ar[rd]_{[1]}\ar[dd]_{[1]} & & M \ar[ll] \\ & w_{\le 0}M   \ar[ru]^{h_0} \ar[ld]^{c_0} & \\ M_0\ar[rr]^{e_{-1}}_{[1]} & & w_{\le -1}M  \ar[lu]_{j_{-1}}\ar[uu]^{h_{-1}} } 
\xymatrix{ w_{\ge 1}M   \ar[dd]_{[1]} & & M \ar[ld] \ar[ll] \\ & w_{\ge 0}M  \ar[lu] \ar[rd]_{[1]}   & \\ M_0 \ar[ru]
  \ar[rr]^{e_{-1}}_{[1]} & & w_{\le -1}M\ar[uu] } 
\end{equation} 

  Next, the first statement in Proposition \ref{pwt}(\ref{iwpt1}) easily follows from Proposition \ref{pbw}(\ref{icompl}). Moreover, for any  $\cu$-morphism 
 $g:M\to M'$   and (any of)  weight Postnikov towers $Po_{\fil_M}$ and $Po_{\fil_{M'}}$ for these object that proposition gives certain morphisms $g_{\le i}$ such that $g\circ h_i=h'_i\circ g_{\le i}$ for all $i\in \z$ 
  (see Definition \ref{dfilt}(3)). 
	 Furthermore, applying the uniqueness part of  Proposition \ref{pbw}(\ref{icompl}) we also obtain 
	 $j'_i\circ g_{\le i} =g_{\le i+1}\circ j_i$; hence   $(g,g_{\le i})$ is  a  morphism $\fil_M\to \fil_{M'}$. It remains to apply Lemma \ref{lrwcomp}(2) to conclude the proof of Proposition \ref{pwt}(\ref{iwpt2}).
 
 Lastly, combining parts \ref{iwpt1} and \ref{iwpt2} of the proposition we obtain that all objects and morphisms of $\cu$ lift to $\pwcu$; this obviously implies Proposition \ref{pwt}(\ref{iwpt3}). 

\ref{iwhefun}. Lemma \ref{lrwcomp}(1,2) implies that we have a  functor $\pwcu\to C(\hw)$ that is compatible with our definition of $t_w$. To prove that 
 $t_w:\cuw\to \kw(\hw)$ is well-defined indeed it clearly remains to verify that in the case $g=0$ the corresponding  $C(\hw)$-morphism $(g^i)$ is weakly homotopy equivalent to $0$.  Thus we can fix $i$ and verify that $g^i$ can be presented in the form $d'{}^{i-1}\circ x^i+y^{i+1}\circ d^i$ for some $x^i\in \hw(M^i,M'{}^{i-1})$ and $y^{i+1}\in \hw(M^{i+1},M'{}^{i})$  (cf. Proposition \ref{pwwh}(\ref{irwc3}) below); 
 here and below we put the "$'$" index in the notation for Postnikov towers 
 to mean that it corresponds to $M'$. 
 Moreover, it obviously suffices to consider the case $i=0$.

Applying Proposition \ref{pbw}(\ref{icompl}) once again  we obtain 
that the composition of $g_{\le 0}$ with $j'_{0}$ is zero. Looking at the distinguished triangle $M_1[-1]\stackrel{e'_{0}[-1]}{\longrightarrow} w_{\le 0}M'\stackrel{j'_{0}}{\to}   w_{\le -1}M'$ 
 we obtain that  $g_{\le 0}$  factorizes as 
 $e'_{0}[-1]\circ a$ for some $a\in \cu( w_{\le 0}M, M'_1[-1])$. Next, applying  Proposition \ref{pbw}(\ref{ifact}) we can present 
 $a$ in the form $x^0\circ c_0$ for some $x^0\in \hw(M^0,M'{}^{-1})$; recall that $M'{}^{-1}=M'_1[-1]$. Thus we have 
$c'_0\circ g_{\le 0}=d'{}^{-1}\circ x^0\circ c_0$. Recall that $c'_0\circ g_{\le 0}=g_0\circ c_0$; hence $(g_0-d'{}^{-1}\circ x^0)\circ c_0=0$. 

Now we apply an essentially dual argument to the morphism $\tg=g_0-d'{}^{-1}\circ x^0$.
Looking at the corresponding $-1$-weight decomposition of $ w_{\le 0}M$ we obtain that 
 $\tg$ factorizes as $b\circ e_{-1}$ for some $b\in \cu((w_{\le -1}M)[1],M'{}^0)$. By  Proposition \ref{pbw}(\ref{ifact}) we can present $b$ as $y^1\circ c_{-1}[1] $ for some $y^1\in \hw(M^1,M'{}^0)$. Thus $g^0=g_0$ can be presented in the form $d'{}^{-1}\circ x^0+y^1\circ d^0$ indeed.



\ref{iwcons}. The assertion easily follows from Theorem 3.1.3(I) of  \cite{bkwn} (see also Proposition 3.1.5(2) of ibid.). However, we  also give another proof here.

The "only if" part of the assertion follows from Proposition \ref{pwt}(\ref{irwcons}) immediately.

Now let us prove the converse implication. It obviously suffices to prove that $M\in \cu_{w\le 0}$ (resp. $M\in \cu_{w\ge 1}$) if  $t(M)$ belongs to $K(\hw)_{\wstu\le 0}$ (resp. to $K(\hw)_{\wstu\ge 1}$) and $M$ is bounded above  (resp. below).
We choose a weight decomposition triangle 
$$LM\stackrel{g}{\to} M\stackrel{f}{\to} RM {\to} LM[1];$$
 consequently,  $LM\in \cu_{w\le 0} $ and $ RM\in \cu_{w\ge 1}$.

Assume that  $M$ is bounded above and $(M^i)=t(M)\in K(\hw)_{\wstu\le 0}$; consider the functor $H=\cu(-,RM)$.
Since this functor kills $\cu_{w\le 0}$ (by the orthogonality axiom of weight structures), Proposition \ref{pwss}(1) gives a convergent weight spectral sequence $T: E_1^{pq}(T)=H^{q}(M^{-p})\implies H^{p+q}(M)$. Next, $H^{q}(M^{-p})=\cu(M^{-p}[-q], RM)=\ns$ if $q\ge 0$. Moreover, the terms $E_2^{pq}(T)$ are completely determined by $t(M)$; hence $E_2^{pq}(T)=\ns$ if $p>0$ (recall the definition of $\wstu$). We obtain that $H(T)=\ns$; thus $f=0$. Hence $M$ is a retract of $LM\in \cu_{w\le 0} $; thus $M$ belongs to  $\cu_{w\le 0} $ itself.

An essentially dual spectral sequence argument  (see Remark \ref{rwcbws}(\ref{isdwss}) below) demonstrates that $g=0$ if $t(M)\in K(\hw)_{\wstu\ge 1}$ and $M$ is bounded below. 

\ref{iwcex}. We fix some choices of weight Postnikov towers for $M$ and $M''$ corresponding to $t(M)$ and $t(M'')$, respectively.   Applying Proposition \ref{pbw}(\ref{iwdext}), for any $i\in \z$  we obtain a commutative diagram 
$$\begin{CD}
w_{\le i}M @>{}>>w_{\le i}M'@>{}>> w_{\le i}M''@>{}>>w_{\le i}M[1]\\
  @VV{}V@VV{}V @VV{}V@VV{}V\\
M@>{g}>>M'@>{f}>>M''@>{}>>M[1]\\
 @VV{}V@VV{}V @VV{}V@VV{}V\\
w_{\ge i+1}M@>{}>>w_{\ge i+1}M'@>{}>>w_{\ge i+1} M''@>{}>>w_{\ge i+1}M[1]\end{CD}
$$
in $\cu$ whose rows and columns are distinguished triangles;  moreover,  the  first three columns are $i$-weight decompositions of  the corresponding objects. 

 Now we apply Proposition \ref{pbw}(\ref{iwdext}) once again and  combine it with assertion  \ref{iwpt1} to obtain commutative diagrams
$$\begin{CD}
w_{\le i-1}M @>{}>>w_{\le i-1}M'@>{}>> w_{\le i-1}M''@>{}>>w_{\le i-1}M[1]\\
@VV{j_{i-1}}V@VV{j'_{i-1}}V @VV{j''_{i-1}}V@VV{j''_{i-1}}V\\
w_{\le i}M @>{}>>w_{\le i}M'@>{}>> w_{\le i}M''@>{}>>w_{\le i}M[1]\\
 @VV{}V@VV{}V @VV{}V@VV{}V\\
M_i@>{g_i}>>M'_i@>{f_i}>>M''_i@>{}>>M_i[1]\end{CD}
$$
here our notation follows Definition \ref{dfilt}. Next, the lower row gives $M'_i\cong M_i\bigoplus M''_i$ immediately from Proposition \ref{pbw}(\ref{isplit}).

Now, by our definitions the families $(g^i)=(g_{-i}[i])$ and $(f^i)=(f_{-i}[i])$ give some choices of $t(g)$ and $t(f)$, respectively. It remains to note that any  composable couple of morphisms of complexes that splits termwisely in each degree gives two sides of a distinguished triangle in $K(\hw)$.

 \ref{iwc2342}. It obviously suffices to consider the case where $g$ and all of the $g^i$ are zero. Consequently, we  assume that $\tilde g{}^i=d'{}^{i-1}\circ x^i+x^{i+1}\circ d^i$ for all $i\in \z$ and 
 some $x^j\in \hw(M^j,M'{}^{j-1})$.

First we verify that $0\in \cu(M,M')$ along with the morphisms $\tilde g_{\le i}=e_i'[-1]\circ x^{-i}[i]\circ c_i$ 
  give a morphism of filtrations in the sense of Definition \ref{dfilt}(3); thus we should check that $j'_i\circ \tilde g_{\le i}=\tilde g_{\le i+1}\circ j_i$ and  $h'_i\circ \tilde g_{\le i}=0$. Now, $ j'_i\circ e_i'[-1]=0$ and $c_{i+1}\circ j_i=0$ (see (\ref{etpt})); hence $j'_i\circ \tilde g_{\le i}= j'_i\circ e_i'[-1]\circ x^{-i}[i]\circ c_i=0=e_{i+1}'[-1]\circ x^{-i-1}[i+1]\circ c_{i+1}\circ j_i =\tilde g_{\le i+1}\circ j_i$.  Next, since $h'_i=h'_{i+1}\circ j'_i$, the composition $h'_i\circ \tilde g_{\le i}$ is zero as well.

Secondly we prove that these  $\tilde g_{\le i}$ along with $\tilde g{}^i$ give a morphism of (weight) Postnikov towers, that is, that the  
 diagram
\begin{equation}\label{e3241} 
  \begin{CD}
 w_{\le i}M@>{c_i}>>M_i @>{e_{i-1}}>>w_{\le i-1}M[1]
\\@VV{\tilde g_{\le i}}V@VV{\tilde g{}^{-i}[i]}V @VV{\tilde g_{\le i-1}[1]}V\\
w_{\le i}M'@>{c'_i}>>M'_i@>{e'_{i-1}}>>w_{\le i-1}M'[1]
\end{CD}\end{equation}
commutes for every $i\in \z$.

Now, $\tilde g{}^{-i}[i]\circ c_i=(d'{}^{-i-1}\circ x^{-i}+x^{1-i}\circ d^{-i})[i]\circ c_i\stackrel{(1)}{=} d'{}^{-i-1}[i]\circ x^{-i}[i]\circ c_i
 = c'_i\circ e_i'[-1]\circ x^{-i}[i]\circ c_i= c'_i\circ \tilde g_{\le i}$; here  to obtain (1) we use the equalities $x^{1-i}[i]\circ d^{-i}[i]\circ c_i= x^{1-i}[i]\circ c_{i-1}[1]\circ e_{i-1}\circ c_i=0$ (see  (\ref{etpt}) once again). 

Lastly, $\tilde g_{\le i}[1]\circ e_{i-1}=e_{i-1}'\circ x^{1-i}[i]\circ c_{i-1}[1]\circ e_{i-1}=e'_{i-1}\circ x^{1-i}[i]\circ d^{-i}[i]=e'_{i-1} \circ \tilde g{}^{-i}[i]$ since $e'_{i-1}\circ d'{}^{-i-1}[i]\circ x^{-i}[i]=e'_{i-1}\circ c'_i\circ e'_i[-1] \circ x^{-i}[i] =0$.
\end{proof}

\begin{rema}
Our definitions  motivate the following question in the context of Proposition \ref{pwt}(\ref{iwc2342}): does  
 the equality $(\tilde g^{i})=(g^i)$ in $\kw(\hw)$ for some $\tilde g^{i}$ such that $(\tilde g{}^i)\in C(\hw)(t(M),t(M'))$ imply that the family $(\tilde g{}^i)$ extends to a morphism of the corresponding weight Postnikov towers that is compatible with $g$?

It appears that the answer to this question is negative in general. However, a 
 simple  modification of the argument above should give the positive answer to it in the case $w=\wstu$ (see Remark \ref{rstws}(1)).
\end{rema}

\subsection{ On 
 problems concerning weight complexes and \cite{bws}}\label{s3bws}

\begin{rema}\label{rwcbws}
\begin{enumerate}
\item\label{irwc52} Our definition of weight complexes is not (quite) self-dual, since for describing the weight complex of $M\in \obj \cu$ in $\cu^{op}$ (with respect to $w^{op}$;  see Proposition \ref{pbw}(\ref{idual})) we have to consider $\{w_{\ge i}M\}$ instead of $\{w_{\le i}M\}$. One may say that there exist "right" and "left" weight complex functors possessing similar properties; note however that the corresponding terms of these complexes (coming from a single weight Postnikov tower) are (non-canonically) isomorphic; see (\ref{doct})  or Proposition 1.5.6(2) of \cite{bws}.

 Moreover,  Remark 1.5.9(1) of ibid. says that these right and left weight complex functors are isomorphic if $\cu$ embeds into a category that possesses a Quillen model.  The general case of this question was not studied in detail yet; possibly, the author will do this in future (and also give a full proof of Proposition 3.2.4  of ibid.; see below). The arguments described in Remark 1.5.9(1) of ibid. may help in studying this self-duality question; yet one probably needs to correct them and  look at (co)homology that is not (co)representable and is related to the proof of  Proposition \ref{pwwh} below.

\item\label{isdwss}
Similarly, the construction of weight spectral sequences in \S2 of \cite{bws} is not self-dual (if we replace $\cu$ by $\cu\opp$, $\au$ by $\au\opp$, and reverse the numeration of terms accordingly) since  the corresponding exact couples come from  weight Postnikov towers.

However, the situation with the properties of these spectral sequences  is much better. Indeed, even though our understanding of weight complexes is not sufficient to compare the corresponding $E_1$-terms, Theorem  \ref{tpure} easily implies that for the "dual" spectral sequence $T'$ we have canonical isomorphisms $E_2^{pq}(T')\cong E_2^{-p,-q}(T)$ for all integers $p$ and $q$. Now, it appears that this statement is sufficient to dualize most of  the applications of weight spectral sequences. 

\item\label{i324} Let us now discuss  flaws in the exposition of the theory of weight complexes 
 in \cite{bws}. 

Firstly, one either has to consider the category $\cuw$ (see Definition \ref{dwpt}(2)) or fix weight Postnikov towers of objects following \S5 of \cite{schnur},  since identifying isomorphic objects in $\kw(\hw)$ (as mentioned in Definition 3.1.6 of \cite{bws}) does not really help to obtain a well-defined functor. However,  
  most  of the proofs from ibid. work in the ("corrected") context of the current paper 
 without any problems. 

So, it appears that the main problem with the arguments of ibid. is Proposition 3.2.4 in that paper, that definitely requires more detail both in the formulation and in the proof  (in particular, its first part depend on a certain self-duality  question). For this reason we mention  new proofs of the statements that depended on it. 

Loc. cit. was applied to the proof of the (nilpotence  statement in) Theorem 3.3.1(II) of ibid. So we note that this statement easily follows  from the more general Theorem 2.3.1(2) of \cite{bkwn} (whose proof depends on  Proposition \ref{pwt}(\ref{iwc2342})); see Remark 2.3.2(1) of ibid. for the detail.

 Next, Proposition 3.2.4 of \cite{bws} was also used in the proof of Theorem 3.3.1(IV) of ibid.; the latter statement is contained in our Proposition \ref{pwt}(\ref{irwcons},\ref{iwcons}). 

We should certainly note that both of these alternative proofs are completely independent from Proposition 3.2.4 of ibid.

\item\label{iiwcex}
We have given the proof of Proposition \ref{pwt}(\ref{iwcex}) above  since the formulation of the  corresponding 
 Theorem 3.3.1(I) of ibid. is somewhat different.

The proofs of the remaining parts of Proposition \ref{pwt} (among those proved in this appendix) 
 is included here for the convenience of the reader mostly. 
 \end{enumerate}
\end{rema}

\section{Some properties of  weak homotopy equivalences} 
 \label{swhe}

Let us prove some properties of the weak homotopy equivalence relation; some of them will be applied elsewhere.

\begin{defn}\label{dbacksim}
 Let  $m_1,m_2:M\to N$ be $C(\bu)$-morphisms (where  $\bu$ is an additive category), $k\le l\in (\{-\infty\}\cup \z\cup \{+\infty\})$, and $k,l\in \z$ if $k=l$.
	
	Then we will write $m_1\backsim_{[k,l]}m_2$ if $m_1-m_2$ is weakly homotopic (see Definition \ref{dwpt}(3)) to  $ m_0\in C(\bu)(M,N)$ such that  $m_0^i=0$ for $k\le i \le l$ (and $i\in \z$).
\end{defn}

\begin{prop}\label{pwwh}
Adopt the notation of Definition \ref{dbacksim}.

\begin{enumerate}
\item\label{irwc4} If $k\in \z$ then $m_1\backsim_{[k,k]}0$ if and only if there exists $m_0\in C(\bu)(M,N)$ such that $m_1=m_0$ in  $K(\bu)(M,N)$ 
  and  $m_0^k=0$.

\item\label{irwc3} Then  $m_1\backsim_{[k,l]}m_2$ if and only if $m_1\backsim_{[i,i]}m_2$ for any $i\in \z$ such that  $k\le i \le l$. 

 Moreover, $m_1$ is weakly homotopic 
 to $m_2$ if and only if $m_1\backsim_{[-\infty,+\infty]}m_2$. 

\item\label{irwid} The morphism class characterized by the condition $g\backsim_{[k,l]}0$ in $C(\bu)$ is a two-sided ideal of morphism, i.e., it is closed with respect to direct sums and if  $g\backsim_{[k,l]}0$ then $ g\circ f \backsim_{[k,l]}0\backsim_{[k,l]} h\circ g$ for any composable morphisms $f,g,$ and $h$.

\item\label{irwsq} If $g$ and $g'$ are composable $C(\bu)$-morphisms and $g\backsim 
0\backsim
 g'$ then the composition $g'\circ g$ is zero as a $K(\bu)$-morphism.

\item\label{iwhecatb}
Factoring morphisms in $K(\bu)$ by the weak homotopy relation yields an additive category $\kw(\bu)$. Moreover, the corresponding full functor $K(\bu)\to \kw(\bu)$ is (additive and) conservative.


\item\label{irwc6}  $M$ belongs to $ K(\bu)_{\wstu\ge 0}$ 
 if and only if $\id_M \backsim_{[1,+\infty]}0_M$,   and $M\in K(\bu)_{\wstu\le 0}$ if and only if  $\id_M \backsim_{[-\infty,-1]}0_M$.

\item\label{iwhefub}
Let $\ca:\bu\to \au$ be an additive functor, where $\au$ is any abelian category, and 
$m_1$ is weakly homotopy equivalent to $m_2$. Then $m_1$ and $m_2$  induce equal morphisms of the homology $H_*(\ca(M^i))\to H_*(\ca(N^i))$.

Hence the correspondence $C\mapsto H_0(\ca(C^{*}))$ gives a well-defined additive functor $\kw(\bu)\to \au$.

\item\label{iwhefud} Assume that $m_1\not \backsim_{[k,k]}0$. Then there exists an additive functor $\ca:\bu\opp\to \ab$ that respects products and such that the corresponding  homomorphism $m_1^*:H_0(\ca(N^{-*}))\to H_0(\ca(M^{-*}))$
is not zero. 
\end{enumerate}
\end{prop}
\begin{proof}
Assertions  \ref{irwc4}  immediately follows from our 
definitions. 

Next, denote by $d^*_{M}$ and $d^*_{N}$ the corresponding boundaries, $M=(M^i)$, $N=(N^i)$. Then
assertion \ref{irwc3} easily follows from the following simple observation: if $k,l\in \z$ and for $m_3=m_1-m_2$ we have $m_3^i=d_N^{i-1}\circ x^i+y^{i+1}\circ d^i_M$ whenever $k\le i\le l$ and  some $x^{i}\in \bu(M^i,N^{i-1})$ and $y^{i+1}\in \bu(M^{i+1},N^{i})$, 
  then $m_0\in C(\bu)(M,N)$, where $m_0^i=0$ for $k\le i\le l$, $m_0^i=m_3^i$ if $i<k-1$ or $i>l+1$, $m_0^{k-1}=x^k\circ d_M^{k-1}$, and $m_0^{l+1}=d_N^{l}\circ y^{l+1}$.

Furthermore, the "direct sum"  part of   assertion \ref{irwid} is obvious, and assertion \ref{irwc3} implies that  it suffices to verify 
 the composition statement in the case $l=k\in \z$. Next we apply assertion \ref{irwc4}  and present $g$ as the $C(\bu)$-sum $g_0+g_1$, where  $g_1$ vanishes  in $K(\bu)$ and  $g_0^k=0$. Then the compositions $g_1\circ f$ and $h\circ g_1$  vanish  in $K(\bu)$ as well, and  $(g_0\circ f)^k=0=(h\circ g_0)^k$. 

Moreover,  a very similar argument easily yields that assertion \ref{iwhefub}  follows from assertions \ref{irwc3} and \ref{irwc4} as well.

Now assume that we have $g\in \cu(M,M')$ and $g'\in \cu(M',M'')$, and  $g^*=d^{*-1}_{M'}\circ x^*+y^{*+1}\circ d^*_{M}$, $g'=d^{*-1}_{M''}\circ x'{}^{*}+y'{}^{*+1}\circ d^*_{M'}$, where 
  $x^*,y^*\in \hw(M^*,M'{}^{*-1})$ and $x'{}^*,y'{}^*\in \hw(M'{}^*,M''{}^{*-1})$ are sequences of morphisms. In the category $K(\bu)$ the sequences of arrows $y^{*+1}\circ d^*_{M}+d^{*-1}_{M'}\circ y^*$ and  $d^{*-1}_{M''}\circ x'{}^{*}+x'{}^{*+1}\circ d^*_{M'}$ yield zero morphisms; thus  $g'\circ g=(y'{}^{*+1}\circ d^*_{M'}\circ d^{*-1}_{M'}\circ x^*)=0$ in this category. 
  Hence we obtain assertion \ref{irwsq}.

\ref{iwhecatb}. Applying assertion \ref{irwid} in the case $k=-\infty$ and $l=+\infty$ we easily obtain the existence of the additive functor in question. 

Next, it is well known  that  a full additive  functor $F$  is conservative 
 whenever the composition of any two (composable) morphisms killed by $F$ is zero. 
 Indeed, it clearly suffices to prove that $g:X\to X$ is an automorphism if $F(g)$ is, and in this case $g(2\id_X-g)=(2\id_X-g)g=\id_X$. 
 Thus assertion \ref{irwsq}  implies that the projection functor $K(\bu)\to \kw(\bu)$  is conservative indeed. 
 
\ref{irwc6}. The "only if" implications are obvious. 
Hence it suffices to verify  the first "if" implication, since the second one is its dual.

If  $\id_M \backsim_{[1,+\infty]}0_M$ then $\id_M$ is weakly homotopic 
 to a morphism $g\in C(\bu)(M,M)$ such that $g^i=0$ for all $i>0$.
Next, $g$ is a $K(\bu)$-automorphism of $M$ according to assertion \ref{iwhecatb}, 
 and it obviously can be factored through the stupid truncation morphism $M \to M^{\le 0}$ (for $M^{\le 0}=\dots \to M^{-1}\to M^0\to 0\to 0\dots$). 
Hence $M$ is a $K(\bu)$-retract of $M^{\le 0}$. Thus $M$ belongs to  $K(\bu)_{\wstu\ge 0}$ indeed (since $K(\bu)_{\wstu\ge 0}$ is  retraction-closed in $K(\bu)$). 

\ref{iwhefud}. 
We take the functor $\ca$ that sends any $B\in \obj \bu$ into $\cok(\bu(B,d^{-1}_N):\bu(B,N^{-1})\to \bu(B,N^0))$. Obviously, this $\ca$ does convert $\bu$-coproducts into products in $\ab$. Thus we should check $m_1^*\neq 0$. 

Now, 
 $H_0(\ca(N^{-*}))$ contains the element $\theta$ corresponding to $\id_{N^0}$. 
 If 
$m_1^*(\theta)=0$ then for 
 $M=(M^i)$ and $m_1=m_1^i$ the element $\theta_M^0\in \ca(M^0)$ corresponding to 
  $m_1^0$ belongs to the image of $\ca(M^1)$ in $\ca(M^0)$.  This obviously implies 
	$m_1\backsim_{[0,0]} 0$. 
\end{proof}

\begin{rrema}\label{rdetkw}
 1. Proposition \ref{pwwh}(\ref{iwhefud}) is closely related to Theorem 2.1 of \cite{barrabs}, and the proof is similar as well. 

2. Clearly, parts \ref{iwhefub} and  \ref{iwhefud} of our proposition  concern the corresponding $\wstu$-pure functors. 
\end{rrema}


\begin{thebibliography}{1}

\bibitem[AKS02]{andkahn} Andr\'e Y., Kahn B., O'Sullivan P., Nilpotence, radicaux et structures mono\"idales // Rendiconti del Seminario Matematico della Universit\'a di Padova, vol. 108, 2002, 107--291.


\bibitem[Ayo17]{ayconj} Ayoub J., Motives and algebraic cycles: a selection of conjectures and open questions,  in:  Hodge theory and $L^2$-analysis, 87--125, Adv. Lect. Math. (ALM), 39, Int. Press, Somerville, MA, 2017.

\bibitem[Ayo18]{ayoubcon} Ayoub J., Topologie feuillet\'ee et la conservativit\'e des r\'ealisations classiques en caract\'eristique nulle, preprint, 2018, \url{http://user.math.uzh.ch/ayoub/PDF-Files/Feui-ConCon.pdf}


\bibitem[Bac17]{bachinv} Bachmann T.,  On the invertibility of motives of affine quadrics// Doc. Math. 22, 2017, 363--395.

\bibitem[BaS01]{bashli} Balmer P.,  Schlichting M., Idempotent completion of triangulated categories// J. of Algebra 236(2),  2001, 819--834. 

\bibitem[Bar05]{barrabs} Barr M., Absolute homology// Theory and Applications of Categories vol. 14,  53--59, 2005.

 \bibitem[BBD82]{bbd} Beilinson A., Bernstein J., Deligne P., Faisceaux pervers// Asterisque 100, 1982, 5--171.

\bibitem[BeV08]{bev} Beilinson A., Vologodsky V., A DG guide to Voevodsky motives//  Geom. Funct. Analysis, vol. 17(6), 2008, 1709--1787.

\bibitem[BlO74]{blog} Bloch S., Ogus A., Gersten's conjecture and the homology of schemes// Ann. Sci. \'Ec. Norm. Sup. 4 ser., v. 7(2), 1974, 181--202.

\bibitem[BoN93]{bokne} B\"okstedt M., Neeman A., Homotopy limits in triangulated categories// Comp. Math. 86, 1993, 209--234. 

\bibitem[Bon10a]{bws} Bondarko M.V., Weight structures vs.  $t$-structures; weight filtrations,  spectral sequences, and complexes (for motives and in general)//  J. of K-theory, v. 6(3), 2010, 387--504; see also \url{http://arxiv.org/abs/0704.4003}

\bibitem[Bon10b]{bger} Bondarko M.V., Motivically functorial coniveau spectral sequences; direct summands of cohomology of function fields// Doc. Math., extra volume: Andrei Suslin's Sixtieth Birthday, 2010, 33--117; see also \url{http://arxiv.org/abs/0812.2672}

\bibitem[Bon11]{bzp} Bondarko M.V., $\mathbb{Z}[\frac{1}{p}]$-motivic resolution of singularities// Comp.\ Math., vol. 147(5), 2011,  1434--1446.

\bibitem[Bon14]{brelmot} Bondarko M.V., Weights for relative motives: relation with mixed complexes of sheaves// Int. Math. Res. Notes, vol. 2014(17), 2014, 4715--4767.  




\bibitem[Bon18a]{bgn} Bondarko M.V., Gersten weight structures for motivic homotopy categories; retracts of cohomology of function fields,  motivic dimensions, and  coniveau spectral sequences, preprint, 2018, \url{https://arxiv.org/abs/1803.01432}


\bibitem[Bon18b]{bcons} Bondarko M.V., Conservativity of realizations implies that numerical motives are Kimura-finite and
motivic zeta functions are rational,  preprint, 2018, \url{https://arxiv.org/abs/1807.10791}

\bibitem[Bon19a]{bkwn}  Bondarko M.V.,  On morphisms killing weights and Hurewicz-type theorems, preprint, 2019, \url{https://arxiv.org/abs/1904.12853} 


\bibitem[Bon19b]{bvtr}  Bondarko M.V., From weight structures to (orthogonal) $t$-structures and back, preprint, 2019, \url{https://arxiv.org/abs/1907.03686}


\bibitem[Bon19c]{bpws}  Bondarko M.V., On perfectly generated weight structures and  adjacent $t$-structures, preprint, 2019, \url{https://arxiv.org/abs/1909.12819}

\bibitem[BoI15]{bonivan} Bondarko M.V., Ivanov M.A., On Chow weight  structures for $cdh$-motives with integral coefficients//  Algebra i Analiz, v. 27(6), 2015, 14--40;  reprinted in St. Petersburg Math. J. 27(6), 2016, 869--888. 



\bibitem[BoK18]{bokum}  Bondarko M.V., Kumallagov D.Z., On Chow weight structures without projectivity and resolution of singularities//   Algebra i Analiz 30(5), 2018, 57--83; transl. in  St. Petersburg Math. J. 30, 2019, 803--819.  





\bibitem[BoS18a]{bos}  Bondarko M.V., Sosnilo V.A., Non-commutative localizations of  additive categories and weight structures;  applications to birational motives//  J. of the Inst. of Math. of Jussieu, vol. 17(4), 2018,  785--821.


\bibitem[BoS18b]{bonspkar} Bondarko M.V., Sosnilo V.A., On constructing weight structures and extending them to idempotent extensions// Homology, Homotopy and Appl., vol. 20(1), 2018, 37--57.

\bibitem[BoS19]{bsnew} Bondarko M.V., Sosnilo V.A., On purely generated $\alpha$-smashing weight structures and weight-exact  localizations// J. of Algebra 535, 2019,  407--455.

\bibitem[BoS20]{bsoscwhn} Bondarko M.V., Sosnilo V.A., On Chow-weight homology of geometric motives, preprint, 2020, 
 \url{https://www.researchgate.net/publication/340849991_On_Chow-weight_homology_of_geometric_motives}.



\bibitem[BoT17]{bontabu}  Bondarko M.V., Tabuada G., Picard groups, weight structures, and (noncommutative) mixed motives// Doc. Math. 22, 2017,   45--66.

\bibitem[BoV19]{bvt}  Bondarko M.V.,  Vostokov S. V., On torsion theories, weight and t-structures in triangulated categories (Russian)//  Vestnik St.-Petersbg. Univ. Mat. Mekh. Astron., vol 6(64), iss. 1, 27--43, 2019; transl. in Vestnik St. Peters. Univers., Mathematics, 2019, vol. 52(1), 19--29. 

\bibitem[Bre67]{bred} Bredon G.E., Equivariant cohomology theories. Lecture Notes in Mathematics 34, Springer, 1967. 

\bibitem[Chr98]{christ} Christensen J., Ideals in triangulated categories: phantoms, ghosts and skeleta// Adv. Math. 136.2, 1998, 284--339.


\bibitem[CiD15]{cdint} Cisinski D.-C., D\'eglise F., Integral mixed motives in equal characteristic//  Doc. Math., Extra Volume: Alexander S. Merkurjev's Sixtieth Birthday, 2015, 145--194.


\bibitem[Deg11]{degmod} D\'eglise F.,  Modules homotopiques (Homotopy modules)//  Doc. Math. 16, 2011, 411--455. 

\bibitem[GiS96]{gs} Gillet H., Soul\'e C., Descent, motives and $K$-theory// J.  f. die reine und ang. Math. v. 478, 1996, 127--176.

\bibitem[Gre92]{green} Greenlees J.P.C., Some remarks on projective Mackey functors // J. of Pure and Appl. Algebra 81(1), 1992, 17--38.

\bibitem[HPS97]{axstab}  Hovey M., Palmieri J., Strickland N., Axiomatic Stable Homotopy Theory, American Mathematical Society, 1997, 114 pp.


\bibitem[Heb11]{hebpo} H\'ebert D., 
Structures de poids \`a la Bondarko sur les motifs de Beilinson// Comp. Math. 147(5), 2011, 1447--1462.


\bibitem[HKS18]{hukriz} Hu P., Kriz I., Somberg P., On some adjunctions in equivariant stable homotopy theory //Algebraic \& Geometric Topology 18(4), 2018, 2419--2442.



\bibitem[KeS17]{kellyweighomol} Kelly S., Saito S., Weight homology of motives// Int. Math. Res. Notices 2017(13), 2017, 3938--3984.

\bibitem[Lew92]{lewishur} Lewis L.G. Jr, The equivariant Hurewicz map// Trans. of the American Math. Soc. 329(2), 1992, 433--472.

\bibitem[LMC86]{rog} Lewis L.G. Jr, May J.P.,  McClure J. E., Ordinary $RO(G)$-graded cohomology// Bull. of the American Math. Soc. 4(2),  1981, 208--212.

\bibitem[LMSC86]{lms} Lewis L.G. Jr, May J.P., Steinberger M.,  McClure J. E., Equivariant stable homotopy theory. Lecture Notes in Mathematics 1213, Springer, 1986.

\bibitem[Mar83]{marg} Margolis H.R., Spectra and the Steenrod Algebra: Modules over the Steenrod Algebra and the Stable Homotopy Category, Elsevier, North-Holland, Amsterdam-New York, 1983.

\bibitem[May96]{mayeq}  May J.P., Equivariant homotopy and cohomology theory: Dedicated to the memory of Robert J. Piacenza,  with contributions by M. Cole, G. Comezana, S. Costenoble, A. D. Elmendorf, J. P. C. Greenlees, L. G. Lewis, Jr., R. J. Piacenza, G. Triantafillou, and S. Waner, CBMS Regional Conference Series in Mathematics 91, American Mathematical Society, 1996.

\bibitem[HVVS13]{mendoausbuch} Hernández O. M., Valadez E. C. S., Vargas V. S.,  Salorio M. J. S., Auslander-Buchweitz context and co-t-structures// Applied Categorical Structures 21(5), 2013, 417--440.


\bibitem[MVW06]{vbook} Mazza C., Voevodsky V.,  Weibel Ch., Lecture notes on motivic cohomology. Clay Mathematics Monographs, vol. 2, 2006.

\bibitem[Nee90]{neederex} Neeman A.,  The derived category of an exact category// J. Algebra 135 (2), 1990, 388--394.

\bibitem[Nee01]{neebook} Neeman A., Triangulated Categories, Annals of Mathematics Studies 148, 2001, Princeton University Press, viii+449 pp.

\bibitem[Pau12]{paucomp}  Pauksztello D., A note on compactly generated  co-t-structures// Comm. in Algebra, vol. 40(2), 2012, 386--394.

\bibitem[Sch11]{schnur} Schn\"urer O., Homotopy categories and idempotent completeness, weight structures and weight complex functors, preprint, 2011, \url{http://arxiv.org/abs/1107.1227}

\bibitem[Sos19]{sosnwc}  Sosnilo V.A., Theorem of the heart in negative $K$-theory for weight structures// Doc. Math.  24,  2019, 2137--2158.

\bibitem[Tho97]{thom} Thomason R.W., The classification of triangulated subcategories// Comp. Math.,  vol. 105(1),  1997, 1--27.


\bibitem[Voe95]{voevnilp} Voevodsky V., A nilpotence theorem for cycles algebraically equivalent to zero// Int. Math. Res. Notices, vol. 1995(4), 1995, 187--199.


\bibitem[Wil15]{wildshim} Wildeshaus J., On the interior motive of certain Shimura varieties: the case of Picard surfaces// Manuscripta Math., vol. 148(3),  2015, 351--377.


\bibitem[Wil18]{wildcons} Wildeshaus J., Weights and conservativity//  Algebraic Geometry 5(6), 2018, 686--702. 
\end{thebibliography}
\end{document}